\documentclass[12pt]{amsart}
\usepackage[margin=1.25in]{geometry}
\usepackage[utf8]{inputenc}
\usepackage{xcolor}
\usepackage{hyperref}
\usepackage{array}    
\usepackage[all]{xy}
\usepackage{placeins}
\usepackage{lscape}
\usepackage{leftidx}
\usepackage{amssymb}
\usepackage[inline]{enumitem}

\usepackage{tikz}
\usepackage{tikz-cd}
\usetikzlibrary{decorations.pathmorphing, knots}
\makeatletter
\newcommand\reallytiny{\@setfontsize\reallytiny{5}{6}}
\makeatother

\def\Gayd{^{\ku \Gamma}_{\ku \Gamma}\mathcal{YD}}

\renewcommand{\Vec}{\text{Vec}}

\newcommand{\qf}{{\mathfrak q}}

\newcommand{\cC}{{\mathcal C}}
\newcommand{\cZ}{{\mathcal Z}}

\newcommand{\Supp}{{\operatorname{supp}}}

\newcommand{\Alg}{{\operatorname{Alg}}}

\newcommand{\ku}{{\Bbbk}}
\newcommand{\Z}{{\mathbb Z}}

\newcommand{\hc}{\mathcal{HC}}

\newcommand{\comod}[1]{^{#1}{\!\mathcal{M}}{}}

\newcommand{\id}{\operatorname{id}}

\newcommand{\cB}{\mathcal{B}}

\newcommand{\cD}{\mathcal{D}}

\newcommand{\Obj}{\mbox{\rm Obj\,}}

\newcommand{\Hom}{\operatorname{Hom}}

\newcommand{\gr}{\mbox{\rm gr\,}}

\theoremstyle{plain}

\numberwithin{equation}{section}

\newtheorem{theorem}{Theorem}[section]

\newtheorem{corollary}[theorem]{Corollary}
\newtheorem{proposition}[theorem]{Proposition}

\newcommand{\xdownarrow}[1]{%
  {\left\downarrow\vbox to #1{}\right.\kern-\nulldelimiterspace}
}

\theoremstyle{definition}
\newtheorem{definition}[theorem]{Definition}
\newtheorem{example}[theorem]{Example}

\theoremstyle{remark}

\newtheorem{remark}[theorem]{Remark}

\begin{document}

\title[Zestings of Hopf Algebras]{Zestings of Hopf Algebras}

\author[I. Angiono]{Ivan Angiono}
\address{Famaf-CIEM (CONICET), Universidad Nacional de C\'ordoba, C\'ordoba, Rep\'ublica Argentina}
\email{ivan.angiono@unc.edu.ar}

\author[C. Galindo]{C\'esar Galindo}
\address{Departamento de Matem\'aticas, Universidad e los Andes, Bogot\'a, Colombia}
\email{cn.galindo1116@uniandes.edu.co}

\author[G. Mora]{Giovanny Mora}
\address{Departamento de Matem\'aticas, Universidad de los Andes, Bogot\'a, Colombia}
\email{hg.mora@uniandes.edu.co}

\begin{abstract}
We extend the previously established zesting techniques from fusion categories to general tensor categories. In particular we consider the category of comodules over a Hopf algebra, providing a detailed translation of the categorical zesting construction into explicit Hopf algebraic terms: we show that the associative zesting of the category of comodules yields a coquasi-Hopf algebra whose comodule category is precisely the zested category. We explicitly write the modified multiplication and the associator, as well as the structures involved in the braided case. 

For pointed Hopf algebras, we derive concrete formulas for constructing zestings and establish a systematic approach for cyclic group gradings, providing explicit parameterizations of the zesting data.
\end{abstract}


\date{\today}

\thanks{I.A. was partially supported by Conicet, SeCyT (UNC) and MinCyT. C.G. was partially supported by Grant INV-2023-162-2830 from the
School of Science of Universidad de los Andes. G.M. was partially supported by Grant INV-2024-187-3052 from the School of Science of Universidad de los Andes. G.M. would like to thank the hospitality and excellent working conditions of the FAMAF at Universidad Nacional de Córdoba, where he carried out part of this research during an academic visit.}
\maketitle

\section{Introduction}

Tensor categories can be thought as axiomatizations of symmetries extending representation theory. 
The profound interplay between tensor categories and 
various areas of mathematics and mathematical physics, including low-dimensional topology, quantum field theory, and quantum computation \cite{EGNO15, BakK2001, Turaev_1994,Wang2010TQC},  has been a driving force behind their development in recent decades.

Hopf algebras provide examples of tensor categories throught the categories of their modules and comodules. Moreover, they often provide the algebraic realization of tensor categories endowed with fiber functors, a connection formalized by Tannakian reconstruction theory \cite{DeligneMilne1982, coquasitriangular1992, JoyalStreet1991_Tannaka,EGNO15}.
The study of Hopf algebras was significantly powered by this fact, along with the emergence of quantum groups \cite{Drinfel1988quantum, Jimbo1985, Lusztig2010, Jantzen1996}. 
Aiming to understand their structures, different classification problems of finite-dimensional Hopf algebras were tackled. One of them was the semisimple case, related to what are called fusion categories. A second one corresponds to pointed Hopf algebras, thought of as generalisations of quantum groups. 
The latter question led to techniques such as the lifting method \cite{AndruskiewitschSchneider2010}.

Within the theory of fusion categories \cite{ENO}, \emph{zesting} has emerged as a construction that provides a method to generate new categories by modifying the monoidal structure (and braiding, if applicable) of an existing $G$-graded category using cohomological data $(\gamma,\lambda, \omega)$ associated with the grading group $G$, \cite{zesting}. Associative zesting alters the tensor product and associativity constraints, while braided zesting further modifies the braiding to produce new braided tensor categories.

Zesting techniques have been used to construct new examples of modular categories and have provided insights into theories with otherwise indistinguishable modular data \cite{DelaneyKimPlavnik2021arxiv}. Although zesting is recognized for its utility as a categorical tool in generating and analyzing these structures, its direct translation into the algebraic framework of general Hopf algebras -extending beyond the semisimple case - motivates further investigation.

In this paper, we develop a  framework for zestings of Hopf algebras, directly extending the established zesting techniques from fusion categories to general tensor categories, specifically focusing on categories of comodules $\comod{H}$ over a Hopf algebra $H$. We show that an associative zesting applied to $\comod{H}$ naturally gives the category of comodules over a coquasi-Hopf algebra \cite{Drinfeld1990_almost, Majid1995Foundations}, denoted $H^\lambda$, whose multiplication comes from appropriately modifying that of  $H$. We explicitly describe $H^\lambda$ by providing  formulas for its modified multiplication, which incorporates the 2-cocycle $\lambda$ from the zesting data, and its non-trivial associator $\Omega$,  determined by the zesting 3-cochain $\omega$ and some additional data. Associative zesting is further extended to the braided setting: if $H$ is coquasitriangular, we show how braided zesting data endows $H^\lambda$ with a coquasitriangular coquasi-Hopf algebra structure, with an explicitly modified $r$-form.

As an application of our general theory, we consider the the class of pointed Hopf algebras. Moreover we concentrate on Hopf algebras of the form $H = \mathcal{B} \# \ku \Gamma$, where $\mathcal{B}$ is a (pre-)Nichols algebra of a Yetter-Drinfeld module $V$ over a group $\Gamma$ \cite{And-Nicholsalgebras}. This focus is motivated because under standard hypotheses (char $\ku=0$, $\ku$ algebraically closed), any finite-dimensional pointed Hopf algebra over an abelian group is a Hopf 2-cocycle deformation of its associated graded Hopf algebra, $\operatorname{gr}H$. The Hopf algebra
$\operatorname{gr}H$ is itself isomorphic to a bosonization of the (pre-)Nichols algebra $\mathcal{B}$ of its primitive elements by the group algebra of its group-like elements $G(H) \cong \Gamma$, i.e., $\operatorname{gr}H \cong \mathcal{B} \# \ku G(H)$ \cite{An13, AnGI}. Given that 2-cocycle deformations induce equivalences between the respective categories of comodules as tensor categories, understanding the zesting of $\comod{\mathcal{B} \# \ku \Gamma}$ provides  insights into the zesting of general finite-dimensional pointed Hopf algebras. For this class, we leverage their universal grading structure, determined by the Hopf cocenter \cite{Andruskiewitsch_Witsch_1996, chirvasitu2016hopf}, to derive concrete formulas for constructing associative zesting data. We establish a systematic approach for zesting with respect to cyclic group gradings, providing explicit parameterizations of the zesting data $(\gamma, \lambda, \omega)$. Our constructions are versatile, covering zesting for pointed Hopf algebras built from braided vector spaces of both diagonal and non-diagonal types. To illustrate the computations techniques, we construct new families of coquasi-Hopf algebras through detailed examples of associative zestings. These include zestings of Nichols algebras of super type $A(1|2)$ (of diagonal type) \cite{AnAn} and the Fomin-Kirillov algebra in three variables (of non-diagonal type).

The paper is organized as follows. Section \ref{section:prelim} provides essential background on tensor categories, Hopf algebras (including coquasi-Hopf and coquasitriangular Hopf algebras), and the categorical zesting construction for $G$-graded tensor categories. In Section \ref{sec:zesting-hopf}, we establish the core of our framework: we discuss the universal grading of Hopf algebras, translate the associative zesting data into Hopf algebraic terms involving relative $r$-forms (related to concepts in \cite{AGP}) and group cocycles, and use Tannakian reconstruction to describe explicitly the resulting coquasi-Hopf algebra $H^\lambda$. This section also addresses braided zesting. Section \ref{sec:examples} focuses on concrete applications, particularly to pointed Hopf algebras. We detail zesting data for pointed Hopf algebras, provide 
explicit constructions for associative zesting by cyclic groups  and present the aforementioned examples, providing new coquasi-Hopf algebras, together with examples 
\cite{bontea2017pointed} in the braided case.

\section{Preliminaries}\label{section:prelim}
In this section, we establish notation and recall the basic definitions and properties of tensor categories, Hopf algebras, and the zesting construction.

\subsection{Tensor categories and Hopf algebras}
Throughout this paper, $\ku$ denotes a field. All vector spaces, (co)algebras, etc. are over $\ku$, and (co)algebra means (co)associative (co)algebra with (co)unity.

\smallbreak 

We assume familiarity with coalgebras and Hopf algebras \cite{M93,R11}. Given a coalgebra $C$, following Sweedler's notation, we write $\Delta(c) = c_{(1)} \otimes c_{(2)}$ for $c \in C$. For iterated coproducts, we write $(\Delta \otimes \id)(\Delta(c)) = c_{(1)} \otimes c_{(2)} \otimes c_{(3)}$, and so on. By $C^{cop}$ we denote the co-opposite coalgebra of $H$, with co-multiplication $\Delta^{cop}(c) = c_{(2)} \otimes c_{(1)}$. We denote by $*$ the convolution product in the dual space $C^*$, defined by $(f * g)(c) = f(c_{(1)})g(c_{(2)})$ for $f,g \in C^*$ and $c \in C$.

\smallbreak 

We also assume familiarity with the basic results of the theory of tensor categories \cite{EGNO15}: We just recall that a tensor category over $\ku$ is a $\ku$-linear abelian category $\cC$ equipped with a bilinear tensor product functor $\otimes: \cC \times \cC \to \cC$, a simple unit object $\mathbf{1}$, and natural isomorphisms for associativity, left and right unit constraints, satisfying coherence conditions and rigidity.

Given a Hopf algebra $H$, the category of finite-dimensional $H$-modules is a tensor category; we will recall above that the category of finite-dimensional 
$H$-comodules is also a tensor category.

A tensor category is \emph{pointed} if all its simple objects are invertible with respect to the tensor product. Given a tensor category $\cC$, we denote by $\cC_{pt}$ the maximal pointed subcategory, which is the subcategory generated by all invertible objects of $\cC$. The set of isomorphism classes of invertible objects in $\cC$ forms a group under the tensor product.

\subsubsection{Coquasi-Hopf algebras}

A \emph{coquasi-bialgebra} $(H,m,u,\omega,\Delta,\varepsilon)$ consists of a coalgebra $(H,\Delta,\varepsilon)$, coalgebra maps $m:H\otimes H\to H$ (the multiplication) and $u:\ku\to H$ (the unit), and a convolution invertible element $\omega\in(H\otimes H\otimes H)^*$ such that
\begin{align*}
(h_{(1)}k_{(1)})l_{(1)}\omega(h_{(2)},k_{(2)},l_{(2)}) &= \omega(h_{(1)},k_{(1)},l_{(1)})h_{(2)}(k_{(2)}l_{(2)}), \\
1_Hh &= h1_H = h, \\
\omega(h_{(1)}k_{(1)},l_{(1)},t_{(1)})\omega(h_{(2)},k_{(2)},l_{(2)}t_{(2)}) &= \omega(h_{(1)},k_{(1)},l_{(1)})  \\
&\qquad \times \omega(h_{(2)},k_{(2)}l_{(2)},t_{(1)})\omega(k_{(3)},l_{(3)},t_{(2)}), \\
\omega(h,1_H,k) &= \varepsilon(h)\varepsilon(k),
\end{align*}
for all $h,k,l,t\in H$.

A \emph{coquasi-Hopf algebra} is a coquasi-bialgebra $H$ with a linear map (called antipode) $S:H\to H$ and elements $\alpha,\beta\in H^*$ such that, for all $h\in H$,
\begin{align*}
S(h_{(1)})\alpha(h_{(2)})h_{(3)} &= \alpha(h)1_H, \\
h_{(1)}\beta(h_{(2)})S(h_{(3)}) &= \beta(h)1_H, \\
\omega(h_{(1)}\beta(h_{(2)}),S(h_{(3)}),\alpha(h_{(4)})h_{(5)}) &= \omega^{-1}(S(h_{(1)}),\alpha(h_{(2)})h_{(3)}\beta(h_{(4)}),S(h_{(5)})) = \varepsilon(h).
\end{align*}

These structures are direct generalizations of Hopf algebras: In fact, when $\omega$ is trivial (that is, $\omega(h,k,l) = \varepsilon(h)\varepsilon(k)\varepsilon(l)$ for all $h,k,l \in H$), and $\alpha=\varepsilon=\beta$, a coquasi-Hopf algebra reduces to an ordinary Hopf algebra.

The group of group-like elements of $H$, denoted by $G(H)$, consists of all non-zero elements $g \in H$ such that $\Delta(g) = g \otimes g$. Let $g_1,g_2 \in G(H)$. An element $h \in H$ is called $(g_1,g_2)$-\emph{skew-primitive} if $\Delta(h) = h \otimes g_1 + g_2 \otimes h$. The set of $(g_1,g_2)$-skew-primitive elements is denoted by $P_{g_1,g_2}(H)$.

\subsubsection{Comodule categories over a coquasi-bialgebra}

Let $H$ be a coquasi-bialgebra. Let $\comod{H}$ denote the category of left $H$-comodules. For $X \in \comod{H}$, we denote its coaction by $\rho_X: X \to H \otimes X$, and write $\rho_X(x) = x_{(-1)} \otimes x_{(0)}$ using again Sweedler's notation. We recall now the monoidal structure of the  category $\comod{H}$. The tensor product of $H$-comodules as vector spaces is the tensor product over $\ku$ with coaction $\rho_{X\otimes Y}(x\otimes y)=x_{(-1)}y_{(-1)}\otimes x_{(0)}\otimes y_{(0)}$. The associativity constraints are:
\begin{align*}
a_{X,Y,Z} : (X \otimes Y) \otimes Z &\longrightarrow X \otimes (Y \otimes Z)\\
x \otimes y\otimes z &\mapsto  \omega(x_{(-1)}, y_{(-1)}, z_{(-1)})x_{(0)} \otimes (y_{(0)} \otimes z_{(0)})
\end{align*}
for $x \in X$, $y \in Y$, $z \in Z$ and $X,Y,Z \in \comod{H}$. The unit object is $\ku$ with $\rho_\ku(1)=1_H\otimes 1$.

By Tannakian reconstruction  \cite{coquasitriangular1992}, a tensor category $\mathcal{C}$ is equivalent to the category of comodules $\comod{H}$ of some (coquasi)-Hopf algebra $H$ if and only if there exists a (quasi)-fiber functor $F: \mathcal{C} \to \Vec_\ku$. Here, a (quasi)-fiber functor is an exact, faithful (quasi)tensor functor from $\mathcal{C}$ to $\Vec_\ku$, where $\Vec_\ku$ denotes the tensor category of $\ku$-vector spaces.

\subsubsection{Braided tensor categories and coquasitriangular Hopf algebras}

A \emph{braiding} on a tensor category $\cC$ is a natural isomorphism
\begin{align*}
c_{X,Y}:& X \otimes Y \to Y \otimes X, && X, Y \in \cC,
\end{align*}
satisfying the hexagon axioms. A \emph{braided tensor category} is a pair consisting of a tensor category and a braiding on it.

A \emph{coquasitriangular} coquasi-bialgebra is a coquasi-bialgebra $H$ with a convolution invertible linear map $r: H \otimes H \to \ku$, called an $r$-form, satisfying:

\begin{align*}
r(h_{(1)},k_{(1)})h_{(2)}k_{(2)} &= k_{(1)}h_{(1)}r(h_{(2)},k_{(2)}),\\
r(h_{(1)}, k_{(1)}l_{(1)})\omega^{-1}(h_{(2)}, k_{(2)}, l_{(2)}) &= \omega(k_{(1)}, l_{(1)},h_{(1)})r(h_{(2)}, l_{(2)})\\
&\quad \times \omega^{-1}(k_{(2)}, h_{(3)},l_{(3)})r(h_{(4)}, k_{(3)}),
\\
r(h_{(1)}k_{(1)},l_{(1)})\omega(h_{(2)}, k_{(2)}, l_{(2)}) &= \omega^{-1}(l_{(1)}, k_{(1)},h_{(1)})r(h_{(2)}, l_{(2)})\\
&\quad \times \omega(h_{(3)}, l_{(3)},k_{(2)})r(k_{(3)}, l_{(4)}),
\end{align*}
for all $h, k, l \in H$.

Given an $r$-form on $H$, we can define a braiding on $\comod{H}$ by:
\begin{align*}
c_{X,Y}: & X \otimes Y \to Y \otimes X, &&x \otimes y \mapsto r(x_{(-1)}, y_{(-1)})y_{(0)} \otimes x_{(0)}.
\end{align*}

\subsection{Zesting of Tensor Categories}

Zesting provides a method to construct new tensor categories from existing ones. Originally defined for fusion categories \cite{zesting}, the construction applies to any tensor category faithfully graded over a group. In this subsection, we recall the basic definitions for strict tensor categories (sufficient for our purpose of studying comodules over Hopf algebras) and some relevant properties for our development of Hopf algebraic zesting.

\subsubsection{Relative Centralizers and the Drinfeld Center}

Before defining the zesting construction, we need to introduce the concept of relative centralizers. 

Let $\cC$ be a tensor category and $\cD \subset \cC$ a tensor subcategory. The \emph{relative centralizer} $R_\cD(\cC)$ is the tensor subcategory of the Drinfeld center $\cZ(\cC)$ whose objects are pairs $(X, \sigma_{X,-})$ where $X \in \cD$ and 
\begin{align*}
\sigma_{X,-} = \{\sigma_{X,Y}: X \otimes Y \to Y \otimes X\}_{Y \in \cC}
\end{align*}
is a family of natural isomorphisms satisfying the half-braiding condition:
\begin{align*}
\sigma_{X, Y \otimes Z} = (\id_Y \otimes \sigma_{X,Z}) (\sigma_{X,Y} \otimes \id_Z)
\end{align*}
for all $Y, Z \in \cC$. 

\subsubsection{Associative zesting}\label{subsection: associative zesting}

With the notion of relative centralizers established, we now define the associative zesting construction for tensor categories.

Let $G$ be a  group. A tensor category $\mathcal{C}$ is $G$-graded if there is a decomposition
\begin{align*}
\mathcal{C} = \bigoplus_{g\in G} \mathcal{C}_g
\end{align*}
of $\mathcal{C}$ into a direct sum of full abelian subcategories such that the tensor product of $\mathcal{C}$ maps $\mathcal{C}_g \times \mathcal{C}_h$ to $\mathcal{C}_{gh}$ for all $g, h \in G$. We will say that the $G$-grading is faithful if $\mathcal{C}_g \neq 0$ for all $g \in G$.

\begin{definition}\label{def:zesting tensor category}
Let $G$ be a group and $\cC = \bigoplus_{g \in G} \cC_g$ a faithfully $G$-graded strict tensor category. An \emph{associative $G$-zesting} for $\cC$ consists of a pair $(\gamma,\lambda, \omega)$ where:
\begin{enumerate}[leftmargin=*]
\item $\lambda: G \times G \to \Obj\left([R_{\cC_e}(\cC)]_{pt}\right)$ is a map such that each $\lambda(g_1, g_2)$ is an invertible object with $\lambda(g_1, e) = \lambda(e, g_1) = \mathbf{1}$ for all $g_1$.
\item For each triple $(g_1, g_2, g_3) \in G^{\times 3}$, an isomorphism 
\begin{align*}
\omega_{g_1,g_2,g_3}: \lambda(g_1,g_2) \otimes \lambda(g_1g_2,g_3) \to \lambda(g_2,g_3) \otimes \lambda(g_1,g_2g_3)
\end{align*}
\end{enumerate}
satisfying $\omega_{g_1,e,g_2} = \id$ and the associative zesting condition for all $g_1, g_2, g_3, g_4 \in G$:
\begin{align}
\begin{aligned}\label{eq:zesting-pentagon}
&(\omega_{g_2,g_3,g_4} \otimes \id_{\lambda(g_1,g_2g_3g_4)}) (\id_{\lambda(g_2,g_3)} \otimes \omega_{g_1,g_2g_3,g_4})((\omega_{g_1,g_2,g_3}) \otimes \id_{\lambda(g_1g_2g_3,g_4)}) \\
&= (\id_{\lambda(g_3,g_4)} \otimes \omega_{g_1,g_2,g_3g_4})(\sigma_{\lambda(g_1,g_2),\lambda(g_3,g_4)} \otimes \id_{\lambda(g_1g_2,g_3g_4)})(\id_{\lambda(g_1,g_2)} \otimes \omega_{g_1g_2,g_3,g_4}).
\end{aligned}
\end{align}
where $\sigma$ denotes the half-braiding associated to objects in $R_{\cC_e}(\cC)_{pt}$.
\end{definition}
\begin{remark}
Definition \ref{def:zesting tensor category} corresponds exactly with \cite[Definition 3.2]{zesting} where equation \eqref{eq:zesting-pentagon} corresponds to \cite[Figure 2]{zesting}. 

\end{remark}

Given an associative $G$-zesting datum $(\gamma,\lambda, \omega)$ for a strict tensor category $\mathcal{C}$, we can construct a new tensor category $\mathcal{C}^{(\gamma,\lambda, \omega)}$. This category has the same underlying  $G$-graded category $\mathcal{C}$, the unit object remains $\mathbf{1}$ but possesses a modified tensor product $\otimes^\lambda$ and a new associativity constraint $a^\lambda$.
\begin{enumerate}[leftmargin=*, label=(\roman*)]
\item The modified tensor product for objects $X\in \mathcal{C}_{g_1}, Y\in \mathcal{C}_{g_2}$ is defined as:
\[
    X \otimes^{\lambda} Y := X \otimes Y \otimes \lambda(g_1, g_2)
\]

\item The associativity constraint
$a^\lambda_{X,Y,Z}: (X\otimes^{\lambda} Y)\otimes^{\lambda} Z \to X\otimes^{\lambda} (Y\otimes^{\lambda} Z)$
for $X\in \mathcal{C}_{g_1}, Y\in \mathcal{C}_{g_2}, Z\in \cC_{g_3}$ is defined by the composition shown in Figure \ref{fig:commutative-diagram-revised}.
\end{enumerate}
\begin{figure}[ht]
\centering
\begin{tikzcd}[column sep=0.5em, row sep=3.5em] 
X \otimes Y \otimes \lambda(g_1,g_2) \otimes Z \otimes \lambda(g_1g_2,g_3)
\arrow[d, "{\mathrm{id}_{X\otimes Y} \otimes \sigma_{\lambda(g_1,g_2), Z} \otimes \mathrm{id}_{\lambda(g_1g_2,g_3)}}"'] \\ 
X \otimes Y \otimes Z \otimes \lambda(g_1,g_2) \otimes \lambda(g_1g_2,g_3)
\arrow[d, "{\mathrm{id}_{X\otimes Y\otimes Z} \otimes \omega_{g_1, g_2,g_3}}"'] \\ 
X \otimes Y \otimes Z \otimes \lambda(g_2,g_3) \otimes \lambda(g_1,g_2g_3)
\end{tikzcd}
\caption{Definition of associator $a^\lambda$ in $\cC^{(\lambda,\omega)}$. }
\label{fig:commutative-diagram-revised} 
\end{figure}

The associative zesting condition \eqref{eq:zesting-pentagon} is precisely what ensures that the newly defined $a^\lambda$ satisfies the pentagon axiom, thereby making $(\mathcal{C}^{(\gamma,\lambda, \omega)}, \otimes^\lambda, \mathbf{1}, a^\lambda)$ a (non-strict) tensor category. The proof given in \cite[Proposition 3.4]{zesting} does not rely on semisimplicity and hence is valid for arbitrary tensor categories.

\subsubsection{Braided zesting}\label{subsubsec: braided zesting}

In this subsection, we discuss the braided zesting construction, drawing from \cite{zesting}. We emphasize that this construction is applicable in non-semisimple settings and for tensor categories graded by groups $G$ which are not necessarily finite. Before presenting the definition of the braided zesting datum, we introduce some relevant concepts concerning tensor automorphisms of the identity functor and their relation to the grading.

For any tensor category $\mathcal{C}$, let $\text{Aut}_\otimes(\text{Id}_\mathcal{C})$ denote the abelian group consisting of tensor natural $\ku$-linear isomorphisms of the identity functor.
Now, assume $\mathcal{C}$ is faithfully graded by a group $G$. Let $\widehat{G} = \Hom(G, \ku^*)$ be the group of linear characters of $G$. Each character $\gamma \in \widehat{G}$ defines a tensor automorphism $\Phi_\gamma \in \text{Aut}_\otimes(\text{Id}_\mathcal{C})$ whose component on any object $X_g$ in the subcategory $\mathcal{C}_g$ (corresponding to grade $g \in G$) is given by scalar multiplication:
\begin{align*}
    (\Phi_\gamma)_{X_g} := \gamma(g)  \id_{X_g}, \quad \text{for } X_g \in \Obj(\mathcal{C}_g).
\end{align*}
The map $\Phi: \gamma \mapsto \Phi_\gamma$ is an injective group homomorphism from $\widehat{G}$ into $\text{Aut}_\otimes(\text{Id}_\mathcal{C})$, where injectivity relies on the faithful grading. We denote the image of this map—the subgroup of tensor automorphisms induced by characters—by $\text{Aut}^{\widehat{G}}_\otimes(\text{Id}_\mathcal{C})$. Natural automorphisms in this subgroup act as constant scalars on the graded components.

Now, let $\mathcal{B}$ be a braided tensor category with braiding $c$.
An invertible object $U \in \mathcal{B}$ will be called \emph{projectively transparent} if there exists a scalar $\lambda_U \in \ku^*$  such that for all objects $Y \in \mathcal{B}$,
\begin{align*}
    c_{Y,U} \circ c_{U,Y} = \lambda_U \id_{U \otimes Y}.
\end{align*}
For any invertible object $U \in \mathcal{B}$, its double braiding also gives rise to a natural isomorphism $\chi_U \in \text{Aut}_{\otimes}(\text{Id}_{\mathcal{B}})$. This $\chi_U$ acts on an object $X \in \mathcal{B}$ by the automorphism $\chi_U(X)$ of $X$, defined via the relation:
\begin{align*}
    \chi_U(X) \otimes \id_U = c_{U,X} \circ c_{X,U}.
\end{align*}
If $\mathcal{B}$ is also faithfully graded by an abelian group $G$, we are interested in cases where an invertible object $U$ is both projectively transparent and its associated $\chi_U$ is compatible with the grading. We say that $\chi_U$ is \emph{character-induced} if it belongs to the subgroup $\text{Aut}^{\widehat{G}}_{\otimes}(\text{Id}_{\mathcal{B}})$. This means there exists a unique character $\gamma_U \in \widehat{G}$ such that $\chi_U = \Phi_{\gamma_U}$. Consequently, for any homogeneous object $X_g$ of grade $g$, the action of $\chi_U$ is purely scalar, determined by its grade:
\begin{align*}
    \chi_U(X_g) = \gamma_U(g) \cdot \id_{X_g}.
\end{align*}

We adopt a common abuse of notation: sometimes $\chi_U(g)$ may denote the tensor automorphism component $\gamma_U(g) \cdot \id_{X_g}$ acting on an unspecified object $X_g$ of grade $g$ (without referring exactly to which object the natural isomorphism is acting), and at other times, $\chi_U(g)$ specifically denotes the scalar value $\gamma_U(g) \in \ku^*$. The intended meaning should be clear from the context.

Consider a braided tensor category $\mathcal{B}$ with braiding $c$, and let $\mathcal{D} \subset \mathcal{B}$ be any tensor subcategory. For any object $X \in \mathcal{D}$, the natural isomorphism $\sigma_{X,-}:=c_{-,X}^{-1}$ provides the structure of an object $(X, \sigma_{X,-})$ in the relative centralizer $R_{\mathcal{D}}(\mathcal{B})$. This yields a canonical braided functor from $\mathcal{D}$ into $R_{\mathcal{D}}(\mathcal{B})$. We utilize this inherent relative braiding structure in the subsequent definition.

\begin{definition}\label{def:cat braided zesting}
Let $\mathcal{B}$ be a strict braided tensor category faithfully graded by an abelian group $G$. A \emph{braided $G$-zesting datum} $(\lambda, \omega, t)$ for $\mathcal{B}$ consists of:
\begin{enumerate}[leftmargin=*]
    \item An associative $G$-zesting datum $(\gamma,\lambda, \omega)$, such that for every $g_1, g_2 \in G$, the automorphism $\chi_{\lambda(g_1,g_2)}$ is character-induced.
    \item A family of isomorphisms $t_{g_1,g_2} : \lambda(g_1, g_2) \longrightarrow \lambda(g_2, g_1)$ for all $g_1, g_2 \in G$.
\end{enumerate}
These data must satisfy the following \emph{braided zesting conditions} for all $g_1, g_2, g_3 \in G$:
\begin{align}
\begin{aligned}
    \omega_{g_2,g_3,g_1}& ( \id_{\lambda(g_2, g_3)} \otimes t_{g_1,g_2g_3})\omega_{g_1,g_2,g_3} \\
    &= (t_{g_1,g_2}\otimes \id_{\lambda(g_1g_2,g_3)}) \omega_{g_2,g_1,g_3}(t_{g_1,g_2}\otimes \id_{\lambda(g_1g_2,g_3)}) \label{eq:bz-cond1-alt} 
\end{aligned}
\end{align}
and, using the abuse of notation $\chi_U(g)$ for the scalar $\gamma_U(g)$:
\begin{align}
\begin{aligned}
    &\chi_{\lambda(g_1,g_2)}(g_3) \cdot \left( \omega_{g_3,g_1,g_2}^{-1}(\id_{\lambda(g_3,g_1)} \otimes t_{g_1g_2,g_3})\omega_{g_1,g_2,g_3}^{-1} \right) \\ 
    &= (t_{g_2,g_3}\otimes \id_{\lambda(g_1,g_2g_3)}) \omega_{g_1,g_3,g_2} (t_{g_1,g_3}\otimes \id_{\lambda(g_3g_1,g_2)})  \label{eq:bz-cond2-alt} 
\end{aligned}
\end{align}
\end{definition}

\begin{remark}
Definition \ref{def:cat braided zesting} is an adaptation of \cite[Definition 4.1]{zesting} where the maps $j$'s in \cite{zesting} are trivial. The condition \cite[(BZ1)]{zesting} is exactly that $\chi_{\lambda(g_1,g_2)}$ is character induced for all $g_1,g_2\in G$. The equation \eqref{eq:bz-cond1-alt} corresponds to \cite[Figure 7]{zesting} and equation \eqref{eq:bz-cond2-alt} corresponds to \cite[Figure 8]{zesting} and this two conditions correspond to \cite[(BZ2)]{zesting}.
\end{remark}

Given a braided $G$-zesting datum $(\lambda, \omega, t)$ for a braided tensor category $\mathcal{B}$, we can equip the associated zested tensor category $\mathcal{B}^{(\gamma,\lambda, \omega)}$ with a braiding, denoted $c^{(t)}$.

For any objects $X \in \mathcal{B}_{g_1}$ and $Y \in \mathcal{B}_{g_2}$, the braiding isomorphism $c^{(t)}_{X,Y}$ is defined as
\begin{align}
    \label{eq:braiding_zesting} 
    c^{(t)}_{X,Y} := c_{X,Y} \otimes t_{g_1,g_2} : X \otimes^\lambda Y \longrightarrow Y \otimes^\lambda X
\end{align}

Again, the proof given in \cite[Proposition 4.4]{zesting} regarding the validity of the hexagon axioms does not rely on semisimplicity and hence holds for arbitrary graded braided tensor categories.

\section{Zesting for Hopf Algebras}\label{sec:zesting-hopf}

In the previous section, we reviewed the general categorical framework for zesting of tensor categories. Now, we specialize this construction to the important setting of categories of comodules over Hopf algebras. Our goal is to translate the categorical zesting data into Hopf algebraic terms. We will show that an associative zesting of the category of comodules $\comod{H}$ leads to a coquasi-Hopf algebra $H^\lambda$ – whose comodule category is precisely the zested category $\comod{H}^\lambda$. We will explicitly describe the modified multiplication and the non-trivial associator  of $H^\lambda$, as well as the structures involved in the braided case.

\subsection{Universal Grading for Hopf Algebras}\label{subsec:universal-grading}

The zesting construction fundamentally relies on equipping the tensor category with a group grading. For fusion categories, there exists a distinguished group, the \emph{universal grading group}, which governs all possible faithful gradings of the category, \cite{nilpotent}.  Importantly, the notion of $U(H)$ also exists for the tensor category $\comod{H}$, where $H$ is any Hopf algebra, without requiring semisimplicity or finite-dimensionality. Its universality stems from the property that any faithful grading of $\comod{H}$ by a group $G$ corresponds to a unique surjective group homomorphism $U(H) \twoheadrightarrow G$. In this subsection we recall the construction of universal grading group in the Hopf algebraic setting.

\begin{definition}
Let $H, K$ be Hopf algebras over $\ku$. A Hopf algebra map $\pi: H \to K$ is called \emph{cocentral} if
\begin{align}\label{eq:cocentral-defn}
    (\mathrm{id} \otimes \pi) \circ \Delta = (\mathrm{id} \otimes \pi) \circ \Delta^{\mathrm{op}}.
\end{align}

The \emph{Hopf cocenter} of $H$, denoted $\hc(H)$, is the maximal cocentral quotient. That is, $\hc(H)$ is characterized by the following universal property: $\hc(H)$ is a quotient Hopf algebra via a cocentral epimorphism $p: H \twoheadrightarrow \hc(H)$ such that any other cocentral Hopf morphism $\pi: H \to K$ factors uniquely through $p$. 
\end{definition}

Every Hopf algebra $H$ admits a unique Hopf cocenter: The existence and uniqueness of $\hc(H)$ were established in \cite{Andruskiewitsch_Witsch_1996}.

The structure of the Hopf cocenter is closely related to the universal grading group of the category $\comod{H}$. Following \cite[\S 3]{chirvasitu2016hopf}:

\begin{definition} Let $H$ be a Hopf algebra.
\begin{enumerate}[leftmargin=*]
    \item The \emph{universal grading group} of $H$, denoted $U(H)$, is the group presented by generators $g_X$ for each isomorphism class of simple left $H$-comodules $X$, subject to the relations:
    \begin{itemize}[leftmargin=*]
        \item $g_X = g_Y$ if $X \cong Y$; 
        \item $g_{\ku} = e$ (the identity element);
        \item $g_X g_Y = g_Z$ if $Z$ is a simple subquotient of the tensor product $X \otimes Y$.
    \end{itemize}
    \item The \emph{group cocenter map} associated with $U(H)$ is the Hopf algebra epimorphism $\pi_{U(H)}: H \twoheadrightarrow \ku U(H)$ corresponding to the functor $\comod{H} \to \comod{\ku U(H)}$ that essentially maps each simple comodule $X$ to the grade $g_X \in U(H)$. 
\end{enumerate}
\end{definition}

The universal property of the Hopf cocenter implies that there is a Hopf algebra epimorphism $\hc(H) \twoheadrightarrow \ku U(H)$. Moreover, the group-like elements capture the grading group structure, as stablished in \cite[Proposition 3.10]{chirvasitu2016hopf}:

\begin{proposition}\label{gradingroup} 
Let $H$ be a Hopf algebra. There exists a group isomorphism between the group of group-like elements $G(\hc(H))$ of the Hopf cocenter $\hc(H)$ and the universal grading group $U(H)$. \qed
\end{proposition}

Let $p: H \twoheadrightarrow \hc(H)$ be the canonical projection. 
Using Proposition \ref{gradingroup} 
we can determine explicitly the components on the grading $U(H) \cong G(\hc(H))$ for any left $H$-comodule $X$: if $\rho: X \to H \otimes X$ denotes the coaction, for each $g \in U(H)$ the $g$-graded component $X_g$ is:
\begin{align*}
    X_g := \{x \in X \mid (p \otimes \mathrm{id})(\rho(x)) = g \otimes x \}.
\end{align*}
Then $X$ decomposes as a direct sum of its graded components: $X = \bigoplus_{g \in U(H)} X_g$. This decomposition 
yields a $U(H)$-grading on the tensor category $\comod{H}$:
\begin{align*}
    \comod{H} = \bigoplus_{g \in G} (\comod{H})_g,
\end{align*}
where $(\comod{H})_g$ is the full abelian subcategory of $H$-comodules $X$ such that $X = X_g$. 

Hence, given any faithful grading of $\comod{H}$ by a group $G$, the universality of $U(H)$ implies the existence of a unique surjective group homomorphism $U(H) \twoheadrightarrow G$. This group homomorphism corresponds precisely to a surjective cocentral Hopf epimorphism $\pi: H \to \ku G$ that induces the $G$-grading on $\comod{H}$. Furthermore, the neutral component $(\comod{H})_e$ of this $G$-grading is  the category of left comodules over the Hopf subalgebra of coinvariants, $H^{\mathrm{co}\pi}$.

\subsection{Associative Zesting for Hopf Algebras}\label{subsec:assoc-zesting-hopf}

Having established possible group grading on $\comod{H}$ in terms of the universal grading group $U(H)$, we now translate the categorical associative zesting data $(\gamma,\lambda, \omega)$ from Subsection \ref{subsection: associative zesting} into concrete Hopf algebraic terms. This involves identifying the algebraic counterparts of the zesting objects $\lambda(g_1, g_2)$ and the associator isomorphisms $\omega$. The key structure connecting the categorical notion of centrality (objects in $R_{(\comod{H})_e}(\comod{H})$) to Hopf algebra maps is the relative $r$-form, which we define first.

\begin{definition}\label{Braided-cotiangular}
Let $H$ be a Hopf algebra and $K$ a Hopf subalgebra of $H$. A \emph{relative $r$-form} for the inclusion $K \subset H$ is a linear map $r: H \otimes K \to \ku$ satisfying the following conditions for all $h, h' \in H$ and $k, k' \in K$:
\begin{align}
r(hh',k)&= r(h',k_{(1)}) r(h,k_{(2)}) \label{1st}, \\
r(h,kk')&=r(h_{(1)},k) r(h_{(2)},k'), \label{2nd} \\
    r(h,1) &= \epsilon(h), \, \, \, r(1,k)= \epsilon(k), \label{3rd} \\
    r(h_{(1)},k_{(1)}) k_{(2)}h_{(2)}&=h_{(1)}k_{(1)} r(h_{(2)},k_{(2)}) \label{4th}
     \end{align}
\end{definition}

As the following result shows, relative $r$-forms precisely capture the data of braided monoidal functors from $\comod{K}$ into the Drinfeld center $\mathcal{Z}(\comod{H})$ compatible  (or equivalently, the relative centralizer $R_{\comod{K}}(\comod{H})$) compatible with the forgetful functor.

\begin{theorem}\label{inclusionequivalences}
Let $H$ be a Hopf algebra and $K$ a Hopf subalgebra with a coquasitriangular structure $s$. The following sets of data are equivalent:
\begin{enumerate}[leftmargin=*, label=(\arabic*)]
\item\label{item:inclusionequivalences-relative-r-matrix} A relative $r$-form $r: H \otimes K \to \ku$ such that $r^{op}|_{K \otimes K} = s$.

\item\label{item:inclusionequivalences-braided-mon-functor} A braided monoidal functor 
\begin{align*}
F : \ \comod{K} \;\longrightarrow\; \mathcal{Z} \left( \comod{H} \right)\;=\; {}_H^H\!\mathcal{YD}
\end{align*}
whose composition with the forgetful functor 
$\mathcal{Z} \left( \comod{H} \right) \to \comod{H}$ is a fully faithful inclusion.

\item\label{item:inclusionequivalences-Hopf-map} A Hopf algebra map $\gamma : K \to (H^\circ)^{\mathrm{cop}}$ such that
\begin{align}\label{eq:inclusionequivalences-Hopf-map-comp-1}
\langle \gamma(k_{(1)}), h_{(1)} \rangle \,k_{(2)} \,h_{(2)}
&= h_{(1)}\,k_{(1)} \,\langle \gamma(k_{(2)}), h_{(2)} \rangle,
& &\text{for all }k \in K, \, h \in H. 
\\ \label{eq:inclusionequivalences-Hopf-map-comp-2}
\langle \gamma(k), k' \rangle &= s(k,k')
& &\text{for all }k,k' \in K. 
\end{align}
Here $H^\circ$ denotes the finite dual of $H$, see \cite{M93} for details. 
\end{enumerate}
\end{theorem}

\begin{proof}
We prove implications cyclically. For \ref{item:inclusionequivalences-relative-r-matrix} $\Rightarrow$ \ref{item:inclusionequivalences-braided-mon-functor}, let $Y \in \comod{K}$. For each $X \in \comod{H}$ we define a natural isomorphism
\begin{align*}
c_{X,Y} &: X \otimes Y \,\longrightarrow\, Y \otimes X, &
c_{X,Y}(x \otimes y) &:= r(x_{(-1)},\,y_{(-1)}) \, y_{(0)} \otimes x_{(0)}.
\end{align*}
Its inverse is given by
\begin{align*}
\overline{c}_{X,Y} &: Y \otimes X \,\longrightarrow\, X\otimes Y, &
\overline{c}_{Y,X}(y \otimes x) &:=
r\bigl(x_{(-1)},\,S(y_{(-1)})\bigr) \, x_{(0)} \otimes y_{(0)}.
\end{align*}
We get then a functor
\begin{align*}
F &: \comod{K} \;\longrightarrow\; \mathcal{Z}(\comod{H}), & 
F(Y) &=\bigl(Y,\;c_{-,Y}\bigr),
\end{align*}
which is a tensor functor whose composition with the usual forgetful functor gives a fully faithful inclusion into $\comod{H}$.

Since $(K,s)$ is coquasitriangular, $\comod{K}$ carries a natural braided tensor structure. As $r^{op}|_{K \otimes K} = s$, $F$ is a braided monoidal functor.

\smallskip

\noindent
\ref{item:inclusionequivalences-braided-mon-functor}  $\Rightarrow$ \ref{item:inclusionequivalences-Hopf-map}
Every $H$-comodule is a colimit of its finite-dimensional subcomodules. Therefore, the image of the monoidal functor
\begin{align*}
\comod{K} \;\longrightarrow\; {}_H^H\!\mathcal{YD} \;\longrightarrow\; {}_H\!\mathcal{M}
\end{align*}  
lies in the subcategory $\underline{{}_H\mathcal{M}}$ of (possibly infinite) colimits of finite-dimensional $H$-modules. Hence we obtain an induced monoidal functor
\begin{align*}
\Gamma: \comod{K} \;\longrightarrow\; {}_H^H\!\mathcal{YD} \;\longrightarrow\;  \comod{(H^*)^{op}},
\end{align*}
which commutes with the corresponding forgetful functors in the following diagram:
\[
\begin{tikzcd}[column sep=2.3em, row sep=1.5em]
   {}^K\!\mathcal{M} 
      \arrow[rr, "\Gamma"]
      \arrow[dr]
   && 
   \comod{(H^*)^{\mathrm{op}}}
      \arrow[dl]
   \\
   & \mathrm{Vec}_{\ku} &
\end{tikzcd}
\]
By Tannakian duality (see \cite[Section~8, Proposition~4]{joyal2006introduction}), there is a unique Hopf algebra map
$\gamma : K \to (H^*)^{\mathrm{op}}$ that realizes $\Gamma$. Concretely, given $X\in\comod{K}$, the action
\begin{align*}
h \,\cdot\, x&= \bigl\langle \gamma(x_{(-1)}),\,h \bigr\rangle\,x_{(0)}, &
h \in H, & \, x \in X,
\end{align*}
recovers the monoidal structure. It remains to check the compatibility conditions: \eqref{eq:inclusionequivalences-Hopf-map-comp-2} follows since $F$ is braided. For \eqref{eq:inclusionequivalences-Hopf-map-comp-1}, fix $X\in\comod{K}$, $x\in X$, $h\in H$.
Using the Yetter–Drinfeld condition and again that $F$ is braided we get
\begin{align*}
h_{(1)} \,x_{(-2)} \,\langle \gamma(x_{(-1)}), h_{(2)}\rangle \otimes x_{(0)}
&= h_{(1)} x_{(-2)} \,\otimes\, \bigl\langle \gamma(x_{(-1)}), h_{(2)} \bigr\rangle\,x_{(0)}
\\
&= h_{(1)} x_{(-1)} \,\otimes\, h_{(2)} \,\cdot\, x_{(0)}
\\
&= (h_{(1)} \cdot x_{(-1)})\,h_{(2)} \,\otimes\, h_3 \cdot x_{(0)}
\\
&= x_{(-1)} \,h_{(2)} \,\otimes\, \bigl\langle \gamma(x_{(-2)}), h_{(1)} \bigr\rangle \,x_{(0)}
  \\[4pt]
  &=
  x_{(-1)}\,h_{(2)}\,\bigl\langle \gamma(x_{(-2)}),h_{(1)}\bigr\rangle \otimes x_{(0)}.
\end{align*}
That is, $h_{(1)} \,x_{(-2)}\,\bigl\langle \gamma(x_{(-1)}), h_{(2)} \bigr\rangle \,\otimes\, x_{(0)} =  x_{(-1)}\,h_{(2)} \,\bigl\langle \gamma(x_{(-2)}),\,h_{(1)} \bigr\rangle \,\otimes\, x_{(0)}$ for all $x \in X$ and $h\in H$.
Taking $X = K$ we get the desired condition.

\smallskip

\noindent
\ref{item:inclusionequivalences-Hopf-map}
$\Rightarrow$ \ref{item:inclusionequivalences-relative-r-matrix}
Finally, given a Hopf algebra map $\gamma : K \to (H^*)^{\mathrm{op}}$, we define
\begin{align*}
r&:H\otimes K\to\ku, &
r(h,k) &:= \langle \gamma(k),\,h\rangle, &
h \in H, & \, k \in K.
\end{align*}
The map $r$ is a relative $r$-form since $\gamma$ is a Hopf algebra map and satisfies \eqref{eq:inclusionequivalences-Hopf-map-comp-1}, and $r^{op}|_{K \otimes K} = s$ by \eqref{eq:inclusionequivalences-Hopf-map-comp-2}.
Thus the three statements are indeed equivalent.
\end{proof}

We can now define the Hopf algebraic version of an associative zesting datum, tailored for the category $\comod{H}$ with an arbitrary faithful $G$-grading. The invertible objects $\lambda(g_1, g_2)$ will correspond to group-like elements in $H$ lying in the neutral component $H_e$ (i.e., elements of $G(H^{\mathrm{co} \pi})$), and the relative centralizer condition will be expressed via the map $\gamma$ from Theorem \ref{inclusionequivalences}.

\begin{definition}\label{defn:assoc-zesting} 
Let $H$ be a Hopf algebra, let $\pi: H\to \ku G$ be a surjective cocentral Hopf homomorphism, and let $\Gamma_0$ be a subgroup of $G(H^{\mathrm{co}\pi})$. An \emph{associative $G$-zesting datum} for $H$ is a triple $(\gamma, \lambda, \omega)$ consisting of:
\begin{enumerate}[leftmargin=*]
\item  A group homomorphism $\gamma: \Gamma_0 \to \Alg(H,\ku)$ such that:
    \begin{align}\label{eq: central condition}
        \langle \gamma(g), h_{(1)} \rangle \, g h_{(2)} g^{-1} &= h_{(1)} \, \langle \gamma(g), h_{(2)} \rangle \quad \text{for all } g \in \Gamma_0, \, h \in H.
    \end{align}
Via Thm. \ref{inclusionequivalences} we get a
relative $r$-form on $\Bbbk G$.
    \item  A normalized 2-cocycle $\lambda: G \times G \to \Gamma_0$.
    \item  A normalized 3-cochain $\omega: G \times G \times G \to \ku^*$ satisfying the compatibility condition:
    \begin{align}\label{assoczesting} 
        \langle \gamma(\lambda(g_1, g_2)), \lambda(g_3, g_4) \rangle &= \frac{\omega(g_1, g_2, g_3) \omega(g_1, g_2g_3, g_4) \omega(g_2, g_3, g_4)}{\omega(g_1g_2, g_3, g_4) \omega(g_1, g_2, g_3g_4)}
    \end{align}
    for all $g_1, g_2, g_3, g_4 \in G$. Here $\langle \gamma(g), g' \rangle$ denotes the evaluation of the algebra map $\gamma(g)$ on the element $g' \in \Gamma_0 \subset H$.
\end{enumerate}
\end{definition}

Here, $\Alg(H,\ku)$ (the set of algebra homomorphisms from $H$ to $\ku$) forms a group under the convolution product. The inverse of an element $f$ in this group is given by $f\circ S$, where $S$ is the antipode of $H$.

Definition \ref{defn:assoc-zesting} provides the Hopf algebraic data $(\gamma, \lambda, \omega)$ corresponding to the categorical data $(\gamma,\lambda, \omega)$ described in Subsection \ref{subsec:assoc-zesting-hopf}. Given this datum, the general construction yields the zested tensor category, which we denote $\comod{H}^{(\gamma, \lambda, \omega)}$.

\subsubsection{Tannakian reconstruction of \texorpdfstring{$H^\lambda$}{H-lambda}}

The zested category $\comod{H}^{(\gamma,\lambda, \omega)}$ has the same underlying abelian category as $\comod{H}$, and its modified tensor product still maps $\comod{H} \times \ \comod{H} \to\ \comod{H}$. It is natural to ask if this new tensor category is the category of comodules of some new algebraic object derived from $H$. The Tannakian reconstruction principle provides the answer.

If $C$ is a coalgebra such that the category $\comod{C}$ is endowed with a tensor product $\otimes'$ and associator $\alpha'$ making it a tensor category (possibly non-strict), and the forgetful functor $(\comod{C}, \otimes') \to \Vec_\ku$ is tensor (quasi-monoidal), then $C$ naturally inherits a (coquasi-)bialgebra structure. Specifically, the multiplication $m'$ and associator $\Omega'$ (as an element in $(C^{\otimes 3})^*$) are defined by:
\begin{align*}
    m'(a \otimes b) &= (a \otimes' b)_{(-1)} \varepsilon\bigl((a \otimes' b)_{(0)}\bigr), &
    \Omega'(a \otimes b \otimes c) &= \varepsilon\bigl(\mathbf{a}_{C,C,C}(a \otimes b \otimes c)\bigr),
\end{align*}
where $a,b,c\in C$ and $\varepsilon$ is the counit. The category of comodules over $(C, m', \Omega')$ is tensor equivalent to $(\comod{C}, \otimes', \alpha')$.

In our case, the zested category $(\comod{H}^{(\gamma,\lambda, \omega)}, \otimes^\lambda, a^\lambda)$ satisfies the conditions for reconstruction. The forgetful functor remains the tensor (it is essentially the identity on the underlying vector spaces). Therefore, the coalgebra $H$ inherits a new structure, yielding a coquasi-bialgebra denoted $H^\lambda$. Its explicit structure is given by the following theorem.

\begin{theorem}\label{coquasi-bialgebra zesting}
Let $H$ be a Hopf algebra, $G$ a group, $H = \bigoplus H_g$ a $G$-grading and $(\gamma, \lambda, \omega)$ an associative $G$-zesting datum for $H$. Then $H$ can be endowed with q structure of coquasi-bialgebra $H^\lambda$ such that:
\begin{itemize}[leftmargin=*]
\item The coalgebra structure of $H^\lambda$ is the same as that of $H$.
\item The multiplication $m^\lambda: H^\lambda\otimes H^\lambda\to H^\lambda$ is given by
\begin{align}\label{eq:mult-zested-co-quasi}
m^\lambda(h \otimes k) &= h k \, \lambda(g_1, g_2), &h \in H_{g_1}, & k \in H_{g_2}.
\end{align}

\item The associator $\Omega \in (H^{\otimes 3})^*$ is given by
\begin{align}\label{eq:assoc-zested-co-quasi}
\Omega(h \otimes k \otimes l) &= \varepsilon(h k) \, \omega(g_1, g_2, g_3) \, \langle \gamma(\lambda(g_1, g_2)), l \rangle,
&  &h \in H_{g_1}, k \in H_{g_2}, l \in H_{g_3}.
\end{align}
\end{itemize}
Furthermore, the category of left $H^\lambda$-comodules is tensor equivalent to the zested category: $ \comod{H^\lambda} \simeq \ \comod{H}^{(\gamma,\lambda, \omega)}$.
\end{theorem}

\begin{proof}
Let $\mathcal{C} =\ \comod{H}$ and consider the zested category $\mathcal{C}^{(\gamma,\lambda, \omega)}$. The tensor product of $X \in \mathcal{C}_{g_1}, Y \in \mathcal{C}_{g_2}$ is $X \otimes^\lambda Y = X \otimes Y \otimes \lambda(g_1, g_2)$. To apply reconstruction, we view this vector space simply as $X \otimes Y$ via the canonical linear isomorphism $X \otimes Y \otimes \ku g \cong X \otimes Y$, where $g=\lambda(g_1,g_2)$, but with a modified coaction $\rho^\lambda$:
\begin{align*}
    \rho^\lambda(x \otimes y) &= (x_{(-1)} y_{(-1)}) \lambda(g_1, g_2) \otimes (x_{(0)} \otimes y_{(0)}).
\end{align*}

Applying the Tannakian reconstruction formula for multiplication we get \eqref{eq:mult-zested-co-quasi}:
\begin{align*}
    m^\lambda(h \otimes k) &= (h \otimes^\lambda k)_{(-1)} \varepsilon((h \otimes^\lambda k)_{(0)}) \\
    &= (h_{(-1)} k_{(-1)}) \lambda(g_1, g_2) \varepsilon(h_{(0)} \otimes k_{(0)}) \\
    &= (h_{(1)} k_{(1)}) \lambda(g_1, g_2) \varepsilon(h_{(2)} k_{(2)}) = hk \lambda(g_1, g_2).
\end{align*}

For the associator $\Omega$, we evaluate the categorical associator $a^\lambda$ from Figure \ref{fig:commutative-diagram-revised} on elements $h \in H_{g_1}, k \in H_{g_2}, l \in H_{g_3}$ and apply the counit $\varepsilon$. Recall $a^\lambda$ involves composition with $\mathrm{id} \otimes \sigma_{\lambda(g_1, g_2), H_{g_3}}$ and $\mathrm{id} \otimes \omega_{g_1,g_2,g_3}$. By Theorem \ref{inclusionequivalences}, the half-braiding $\sigma_{\lambda, l}$ acts as $l' \otimes l \mapsto l_{(0)} \otimes l' \langle \gamma(l'), l_{(-1)} \rangle$. Applying $\varepsilon \circ a^\lambda$ on $h \otimes k \otimes l$:
\begin{align*}
    \Omega(h\otimes k \otimes l) &= \varepsilon( a^\lambda(h \otimes k \otimes l) ) \\
      &= \varepsilon(hk) \, \omega(g_1, g_2, g_3) \, \langle \gamma(\lambda(g_1, g_2)), l \rangle,
\end{align*}
so \eqref{eq:assoc-zested-co-quasi} holds.
\end{proof}

\subsection{Braided Zesting for Hopf Algebras}\label{subsec:braided-zesting-hopf}

Finally, we consider braided zesting in the Hopf algebra context. This requires starting with a Hopf algebra $H$ whose category of comodules $\comod{H}$ is already braided.

Recall that $\comod{H}$ is braided if and only if $H$ is \emph{coquasitriangular}, see  \cite{M93,radford2011hopf} for details. This is equivalent to the existence of a convolution-invertible linear map $r: H \otimes H \to \ku$ (the universal $r$-form) satisfying certain axioms, or equivalently, the existence of a Hopf algebra map $f: H \to (H^\circ)^{\mathrm{cop}}$ (compatible with the opposite coproduct on the target) such that
\begin{align*}
    \langle f(k_{(1)}), h_{(1)} \rangle \, k_{(2)} h_{(2)} &= h_{(1)} k_{(1)} \, \langle f(k_{(2)}), h_{(2)} \rangle, \quad \forall h, k \in H.
\end{align*}
The $r$-form is recovered via $r(h, k) = \langle f(k), h \rangle$. The braiding on $\comod{H}$ is then given by $c_{X,Y}(x \otimes y) = r(x_{(-1)}, y_{(-1)}) y_{(0)} \otimes x_{(0)}$.
The alternative braiding $c'_{X,Y} := c_{Y,X}^{-1}$ on $\comod{H}$ corresponds to the $r$-form $r'(h, k) := r^{-1}(k, h)$ for all $h, k \in H$, where $r^{-1}: H \otimes H \to \ku$ is the convolution inverse of $r$.
For the subsequent discussion, we highlight the algebra map associated with $r'$. This map, denoted \emph{$f': H \to (H^{\circ})^{\mathrm{cop}}$}, is defined by the relation
$$ \langle f'(k), h \rangle = r'(h,k) = r^{-1}(k, h) \quad \text{for all } h, k \in H. $$
The map \emph{$f'$} is important for Definition~\ref{def: braided zesting for Hopf}.

The following definition characterizes group-like elements $g \in G(H)$ in a coquasitriangular Hopf algebra $H$ for which the associated natural isomorphism $\chi_g$ is character-induced with respect to a $G$-grading (where the grading is induced by a cocentral epimorphism $\pi: H \to \Bbbk[G]$).

\begin{definition}\label{def:pi_character_inducing_grouplike}
Let $(H,r)$ be a coquasitriangular Hopf algebra and let $\pi: H \to \Bbbk[G]$ be a cocentral Hopf epimorphism. A group-like element $g \in G(H)$ is called \emph{$\pi$-character-inducing} if there exists a group character $\lambda_g: G \to \Bbbk^*$ such that \begin{align*}
r(x_{(1)},g)r(g,x_{(2)}) = \lambda_g(s)\epsilon(x).
\end{align*}
for every $s \in G$ and for every  $x\in H_s := \{h \in H \mid (\pi \otimes \id_H)\Delta(h) = s \otimes h\}$.
\end{definition}

We now define the data needed for braided zesting, building upon the associative datum $(\gamma, \lambda, \omega)$ where $f'|_{G(H^{co\pi})}= \gamma$.

\begin{definition}\label{def: braided zesting for Hopf}
Let $(H,f)$ be a coquasitriangular Hopf algebra and $\pi: H \rightarrow \ku[G]$ a surjective cocentral Hopf homomorphism with $G$ an abelian group. A \emph{braided $G$-zesting} $(\lambda, \gamma, \omega, t)$ corresponds to:
\begin{enumerate}[leftmargin=*, label=(\arabic*)]
\item An associative $G$-zesting $(\lambda, \gamma, \omega)$ such that:
\begin{itemize}
\item $f'|_{G(H^{co\pi})}= \gamma$,
\item $\lambda(g_1,g_2)$ is $\pi$-character-induced,
\item The 2-cocycle is symmetric, i.e., $\lambda(g_1,g_2) = \lambda(g_2,g_1)$ for all $g_1,g_2 \in G$.
\end{itemize}
\item A 2-cochain $t\in C^2(G,\mathbb{C}^*)$ satisfying the following conditions:
\begin{align}\label{eq:BZ2}
\frac{\omega(g_1, g_2, g_3) \omega(g_2, g_3, g_1)}{\omega(g_2, g_1, g_3)} &= \frac{t(g_1, g_2) t(g_1, g_3)}{t(g_1, g_2 g_3)},
\\
\label{eq:BZ3}
r(\lambda(g_1,g_2), g_3) r(g_3, \lambda(g_1,g_2)) &= \frac{\omega(g_1, g_2, g_3) \omega(g_3, g_1, g_2)}{\omega(g_1, g_3, g_2)} \frac{t(g_1, g_3) t(g_2, g_3)}{t(g_1 g_2, g_3)},
\end{align}
for all $g_1, g_2, g_3 \in G$.
\end{enumerate}
\end{definition}

Just as associative zesting led to the coquasi-Hopf algebra $H^\lambda$, incorporating the braided zesting datum $(\gamma, \lambda, \omega, t)$ endows $H^\lambda$ with a braided structure.

\begin{corollary}
Given a braided $G$-zesting datum $(\gamma, \lambda, \omega, t)$ for a coquasitriangular Hopf algebra $(H, r)$, the resulting coquasi-Hopf algebra $H^\lambda$ from Theorem \ref{coquasi-bialgebra zesting} is also braided coquasitriangular. Its universal $r$-form $r^\lambda$ is obtained by twisting $r$ with the cochain $t$; explicitly,
\begin{align*}
r^\lambda(x, y) &= t(g_1, g_2) r(x, y), & x \in H^\lambda_{g_1}, & y \in H^\lambda_{g_2}.
\end{align*}
Consequently, the category $\comod{H^\lambda}$, with the braiding induced by $r^\lambda$, is braided tensor equivalent to the zested category obtained from $\comod{H}$.\qed
\end{corollary}


\section{Computing Examples: Zesting Pointed Hopf Algebras}\label{sec:examples}

In this section, we apply zesting construction to the important class of pointed Hopf algebras. Next we develop explicit formulas for associative zestings when the grading group is cyclic. Combining these two constructions we provide concrete examples of zestings for pointed Hopf algebras graded by cyclic groups.

\subsection{Pointed Hopf algebras}

Recall that a Hopf algebra $H$ is \emph{pointed} if all its simple comodules are one-dimensional. Equivalently, its coradical $H_0$ is the group algebra $\ku G(H)$ of its group-like elements $G(H)$. The coradical filtration 
\begin{align*}
&H_0 \subset H_1 \subset H_2 \subset \cdots, & \text{where }& H_n = \{h \in H \mid \Delta(h) \in H_0 \otimes H+H \otimes H_{n-1}\},
\end{align*}
is a filtration by Hopf subalgebras.  The associated graded object $\gr H = \bigoplus\limits_{n \geq 0} H_n/H_{n-1}$ is canonically a (coradically) graded Hopf algebra.

Assume that $\ku$ is algebraically closed of characteristic zero. If $H$ is finite-dimensional and $G(H)$ is abelian, then $H$ is generated by group-like and skew-primitive elements \cite{An13}. Furthermore, $\gr H$ takes the specific form of a bosonization 
\begin{align*}
\gr H &\cong \mathcal{B}(V) \# \ku \Gamma, &\text{where }&\Gamma=G(H).
\end{align*}
Moreover, $V = H_1/H_0$ is the space of primitive elements of $\gr H$, viewed as a Yetter-Drinfeld module over $\Gamma$, and $\mathcal{B}(V)$ is the Nichols algebra of $V$. Nichols algebras are graded Hopf algebras in the braided fusion category $\Gayd$ of Yetter-Drinfeld modules, defined as quotients of the tensor algebra $T(V)$ by a specific maximal homogeneous ideal, see \cite{And-Nicholsalgebras} for a precise definition. Crucially, under these assumptions, $H$ is known to be a cocycle deformation of $\gr H$ \cite{AnGI}, implying their categories of comodules are tensor equivalent: $\comod{H} \simeq_\otimes \ \comod{\gr H}$.

This relationship holds for all known examples even when $G(H)$ is not abelian; that is, any finite-dimensional pointed Hopf algebra $H$ is a cocycle deformation of $\gr H \cong \mathcal{B}(V) \# \ku G(H)$. Hence, understanding zesting for the bosonized form $\mathcal{B}(V) \# \ku G(H)$  is sufficient for all known examples of finite dimensional pointed Hopf algebras.



\subsection{Associative Zestings for Pointed Hopf Algebras}\label{subsec:assoc-zesting-pointed}

Our first goal is to determine associative zesting data for $\comod{H}$ where $H$ is a finite-dimensional pointed Hopf algebra. Based on the discussion above, it suffices to analyze Hopf algebras of the form $H = \mathcal{B}(V) \# \ku \Gamma$, where $\Gamma$ is a finite group, $V \in {}_{\Gamma}^{\Gamma}\mathcal{YD}$ is a Yetter-Drinfeld module such that the Nichols algebra $\mathcal{B}(V)$ is finite-dimensional.

We will work in a slightly more general context: We fix a Hopf algebra $H=\cB\# \ku \Gamma$, not necessarily finite-dimensional, where $\Gamma$ is a group  and $\cB$ is a pre-Nichols algebra of $V\in\Gayd$ with $\dim V<\infty$, that is, an intermediate graded Hopf algebra between $\cB(V)$ and $T(V)$ in the braided tensor category $\Gayd$. 

\smallbreak

As $\cB$ is a quotient of the tensor algebra $T(V)$, $\cB$ is generated by $V$ as an algebra. We may fix a basis $(x_i)_{1\le i\le \theta}$ of $V$ such that $x_i$ has coaction given by some $g_i\in \Gamma$, so $\Delta(x_i)=x_i\otimes 1 + g_i \otimes x_i$ in $H$. Thus $H$ is a graded Hopf algebra, generated as an algebra by the group-like elements $g\in \Gamma$ (in degree zero) and the skew-primitive elements $(x_i)_{1\le i\le \theta}$ (in degree one). The subgroup $\langle g_i \rangle$ of $\Gamma$ generated by $g_i$,  $1\le i\le \theta$, is called the \emph{support} of $V$ and denoted by
$\Supp (V)$. 
The Yetter-Drinfeld compatibility condition says that $\Supp (V)$ is a normal subgroup of $\Gamma$.

For this class of Hopf algebras, the universal grading group has a simple description.

\begin{proposition}\label{Universal grading pointed}
Let $H = \mathcal{B} \# \ku \Gamma$ be constructed as above. Then the universal grading group is isomorphic to the quotient group:
$$ U(H) \cong \Gamma / \Supp(V). $$
\end{proposition}

\begin{proof} 

Let $\varpi:H\twoheadrightarrow \ku \tfrac{\Gamma}{\Supp (V)}$ be the composition of the projection $H\twoheadrightarrow \ku\Gamma$ onto the degree one component with the canonical projection $\ku\Gamma \twoheadrightarrow \ku \tfrac{\Gamma}{\Supp (V)}$. Then $\varpi$ is a Hopf algebra map since it is the composition of two of them. Notice that
\begin{align*}
(\id \otimes \varpi) \circ \Delta(g) &= g\otimes \varpi(g)=(\id \otimes \varpi) \circ \Delta^{op}(g) && \text{for all }g\in \Gamma 
\\
(\id \otimes \varpi) \circ \Delta(x_i) &= x_i\otimes 1 = (\id \otimes \varpi) \circ \Delta^{op}(x_i) && \text{for all }i=1,\cdots,\theta.
\end{align*}
As $x_i$, $g$ generate $H$ as an algebra and $\Delta$ and $\varpi$ are algebra maps,
\begin{align*}
(\id \otimes \varpi) \circ \Delta(x) &= (\id \otimes \varpi) \circ \Delta^{op}(x) && \text{for all }x\in H; 
\end{align*}
that is, $\varpi$ is cocentral.

Let $K$ be a Hopf algebra and $\pi:H \to K$ a cocentral Hopf morphism. Applying \eqref{eq:cocentral-defn} to each $x_i \in H$, we have:
$$ x_i\otimes 1 + g_i \otimes \pi(x_i)= 1 \otimes \pi(x_i) + x_i \otimes \pi(g_i). $$
Hence, $\pi(g_i)=1$ and $\pi(x_i)=0$ for all $1\le i\le\theta$. Thus, the map $\pi$ factors throught $\varpi$. By the universality of the Hopf cocenter, $\hc \simeq \Bbbk\tfrac{\Gamma}{\Supp (V)}$. Hence, $U(H)\simeq \tfrac{\Gamma}{\Supp (V)}$, by Proposition \ref{gradingroup}.
\end{proof}



Now fix $H$ as above and $\pi: H\to \ku G$ a cocentral Hopf epimorphism. 
From the proof of Proposition \ref{Universal grading pointed}, we have that  $G(H^{co\pi}) = \Supp (V)$ and the restriction of $\pi$ to $H_0=\ku\Gamma$ gives a group epimorphism $\Gamma\twoheadrightarrow G$, also called $\pi$.

The next result will allow to find explicit computations for associative zestings:

\begin{proposition}\label{prop:relative_r_pointed}

Let $H = \mathcal{B} \# \ku \Gamma$ be constructed as above, and let $V=\ku\{x_i\}_{1\leq i\leq \theta}$ be the associated $\Gamma$-Yetter-Drinfeld module.

If $\Gamma_0$ is a subgroup of $\Gamma$ such that the Hopf subalgebra $\ku\Gamma_0 \subseteq H$ admits a relative $r$-form, then $\Gamma_0$ must be a central subgroup of $\Gamma$ and must satisfy $\Gamma_0 \subseteq \Supp(V)$.

Moreover, if $\Gamma_0$ is a central subgroup of $\Gamma$ such that $\Gamma_0 \subseteq \Supp(V)$, then there is a bijective correspondence between:
\begin{enumerate}[leftmargin=*, label=(\roman*)]
\item Relative $r$-forms $r: H \otimes \ku \Gamma_0 \to \ku$.
\item Group homomorphisms $\Phi: \Gamma_0 \to \widehat{\Gamma}$ such that
\begin{align}\label{eq:central-subgr-condition}
g \cdot x_i &= \langle g_i, \Phi(g) \rangle x_i && \text{for all } g \in \Gamma_0 \text{ and } 1 \le i \le \theta.
\end{align}
\end{enumerate}

\end{proposition}

\begin{proof}
It is analogous to \cite[Proposition 4.3]{AGP}, using Theorem \ref{inclusionequivalences}.
\end{proof}

\begin{remark}
Condition \eqref{eq:central-subgr-condition} can be rewritten only in terms of the multiplication in the algebra $H = \mathcal{B} \# \ku \Gamma$. The Yetter-Drinfeld action of an element $g \in \Gamma$ on $x_i \in V$ is given by $g \cdot x_i = g x_i g^{-1}$ (where the product $g x_i g^{-1}$ is taken in $H$). Thus, for $g \in \Gamma_0$, the condition becomes
\begin{align*}
g x_i g^{-1} &= \langle g_i, \Phi(g) \rangle x_i && \text{for all } g \in \Gamma_0 \text{ and } 1 \le i \le \theta.
\end{align*}
The elements $x_i$ are skew-primitive in $H$. 

\end{remark}
We now specialize the associative Hopf zesting to pointed Hopf algebras of the form $H = \mathcal{B} \# \ku \Gamma$. Incorporating the specific results for their universal grading group $U(H) \cong \Gamma/\Supp(V)$ and the characterization of relative $r$-form (via Proposition \ref{prop:relative_r_pointed}), the following theorem details the necessary zesting data $(\gamma, \lambda, \omega)$ and explicitly describes the resulting coquasi-Hopf algebra structure $H^\lambda$. 

\begin{theorem}\label{thm:pointed_zesting_datum}
Let $H = \mathcal{B} \# \ku \Gamma$ be a pointed Hopf algebra. An associative $G$-zesting datum $(\gamma, \lambda, \omega)$ for $H$ is determined by the following choices:
\begin{enumerate}[leftmargin=*]
    \item A grading group $G$ chosen via a surjective group homomorphism 
    \begin{align*}
    \iota\colon \Gamma/\Supp(V) \to G.
    \end{align*}
    
    \item A pair $(\Gamma_0, \Phi)$, where $\Gamma_0 \leq \ker(\iota)$  and 
    $\Phi: \Gamma_0 \to \widehat{\Gamma}$ satisfies \eqref{eq:central-subgr-condition}, thus determining relative $r$-form via the  map $\gamma\colon \Gamma_0 \to \Alg(H,\ku)$.
    
    \item A normalized 2-cocycle $\lambda\colon G \times G \to \Gamma_0$.
    
    \item A normalized 3-cochain $\omega\colon G \times G \times G \to \ku^{\times}$ satisfying \eqref{assoczesting}.
\end{enumerate}

Given such a valid datum $(\gamma, \lambda, \omega)$, the resulting coquasi-Hopf algebra $H^{\lambda}$ shares the same underlying coalgebra structure as $H$. Its multiplication $m^{\lambda}$ and associator $\Omega$ are given by:
\begin{align}
    m^{\lambda}((x \# g) \otimes (y \# h)) &= x (g \cdot y) \# gh \lambda(\iota(g), \iota(h)), \label{eq:pointed_m_lambda_v4}\\
    \Omega(x \# g, y \# h, z \# k) &= \varepsilon(x)\varepsilon(y)\varepsilon(z) \, \omega(\iota(g), \iota(h), \iota(k)) \, \langle \gamma(\lambda(\iota(g), \iota(h))), k \rangle, \label{eq:pointed_Omega_v4}
\end{align}
where $x, y, z \in \mathcal{B}$; $g, h, k \in \Gamma$; $\varepsilon$ is the counit; and the arguments of $\lambda$ and $\omega$ are the images in $G$ via $\iota$.  
\end{theorem}

\subsection{Associative zestings of Hopf algebras by cyclic groups} \label{subsec:cyclic_zesting}

We will construct explicit associative zestings when the grading group  is cyclic. This construction leverages standard representatives for the second cohomology group $H^2(\mathbb{Z}/N\mathbb{Z}, \Gamma_0)$ and specific cochains whose properties simplify the verification of the zesting conditions.

\subsubsection{Background on specific cyclic cochains.}
Let $G = \mathbb{Z}/N\mathbb{Z}$. We identify the elements of $G$ with $\{0, 1, \dots, N-1\}$. Let $M$ be an abelian group (written multiplicatively). Consider the following cochains associated with $\nu \in M$: 
\begin{align}
    \beta_\nu(i) &= i\nu \label{eq:cyclic_beta_nu_def_final}\\
    \lambda^{(\nu)}(i, j) &=
    \begin{cases}
        1 & \text{if } i+j < N \\
        \nu & \text{if } i+j \ge N
    \end{cases} \quad &&(\nu \in M, i,j \in G) \label{eq:cyclic_lambda2_nu_def_final}\\
    \omega^{(\nu)}(i, j, k) &=
    \begin{cases}
        1 & \text{if } i+j < N \\
        \nu^k & \text{if } i+j \ge N
    \end{cases} \quad &&(\nu \in M, i,j,k \in G) \label{eq:cyclic_omega3_nu_def}\\
    \theta^{(\nu)}(i, j, k, l) &=
    \begin{cases}
        \nu & \text{if } i+j \ge N \text{ and } k+l \ge N \\
        1 & \text{otherwise.}
    \end{cases} \quad &&(\nu \in M, i,j,k,l \in G) \label{eq:cyclic_theta4_nu_def_final}
\end{align}
Let $\delta$ be the standard group cohomology differential (trivial action). From \cite{zesting} we recall some key relations between these cochains:
\begin{align}
    \delta(\beta_\nu) &= \lambda^{(N\nu)} \label{eq:rel_delta_beta_final}\\ 
    \delta(\lambda^{(\nu)}) &= 0 \label{eq:rel_delta_lambda2_final}\\ 
    \delta(\omega^{(\nu)}) &= \theta^{(N\nu)} \label{eq:rel_delta_omega3_final} 
\end{align}
From \eqref{eq:rel_delta_lambda2_final}, $\lambda^{(\nu)} \in Z^2(G, M)$ and from \eqref{eq:rel_delta_beta_final}, $\lambda^{(N\nu)} \in B^2(G, M)$).
Let $M_N := \{ m \in M \mid m^N = 1 \}$. Given $\nu \in M_N$, \eqref{eq:rel_delta_omega3_final}  says that $\omega^{(\nu)} \in Z^3(G, M)$.

It is well known that these cochains represent the cohomology groups via the induced homomorphisms:
\begin{align}
    \nu + NM \mapsto [\lambda^{(\nu)}]: M/NM \xrightarrow{\cong} H^2(G, M), \quad \nu \mapsto [\omega^{(\nu)}]: M_N \xrightarrow{\cong} H^3(G, M). \label{eq:cohom_isos_final}
\end{align}

\subsubsection{Associative cyclic zesting.}

The preceding discussion on cyclic cochains allows us to construct a concrete associative zesting when $G=\mathbb{Z}/N\mathbb{Z}$. We use the standard 2-cocycle $\lambda^{(\nu)}$ for $\lambda$. The following proposition identifies the corresponding 3-cochain $\omega^{(q)}$ that satisfies the condition \eqref{assoczesting}.

\begin{proposition}\label{prop:cyclic_zesting}
Let $H$ be a Hopf algebra and let $\pi: H \to \ku \mathbb{Z}/N\mathbb{Z}$ be a surjective cocentral Hopf epimorphism. We fix the following data:
\begin{enumerate}[leftmargin=*]
\item an abelian subgroup $\Gamma_0 \le G(H^{\mathrm{co}\pi})$,
\item a group homomorphism $\gamma: \Gamma_0 \to \Alg(H,\ku)$ satisfying \eqref{eq: central condition},
\item an element $\nu \in \Gamma_0$, and define the scalar $m = \langle \gamma(\nu), \nu \rangle \in \ku^*$.
\end{enumerate}
For any $q \in \ku^*$ such that $q^N = m$, the triple $(\gamma, \lambda^{(\nu)}, \omega^{(q)})$ forms an associative $G$-zesting datum for $H$, where $\lambda^{(\nu)}$ and $\omega^{(q)}$ are the standard cyclic cochains defined in \eqref{eq:cyclic_lambda2_nu_def_final} and \eqref{eq:cyclic_omega3_nu_def} respectively.
\end{proposition}

\begin{proof}
We verify the condition \eqref{assoczesting}. 
The left hand side of \eqref{assoczesting} is $\langle \gamma(\lambda(i, j)), \lambda(k, l) \rangle$. Substituting $\lambda=\lambda^{(\nu)}$ yields $m = \langle \gamma(\nu), \nu \rangle$ if $i+j \ge N$ and $k+l \ge N$, and 1 otherwise. Thus, $\langle \gamma(\lambda(i, j)), \lambda(k, l) \rangle=\theta_m(i,j,k,l)$. The right hand side is $\delta (\omega^{(q))}$. Using the relation $\delta(\omega^{(q)}) = \theta_{q^N}$, we get 
\begin{align*}
\langle \gamma(\lambda(i, j)), \lambda(k, l) \rangle=\theta_m(i,j,k,l)=\theta_{q^N}(i,j,k,l)=\delta(\omega^{(q)})(i,j,k,l),
\end{align*}
that is exactly \eqref{assoczesting}.
\end{proof}

\begin{remark}
To apply Proposition \ref{prop:cyclic_zesting}, one selects an element $\nu \in \Gamma_0$. Since the cohomology class $[\lambda^{(\nu)}] \in H^2(\mathbb{Z}/N\mathbb{Z}, \Gamma_0)$ depends only on the coset $\nu \Gamma_0^N$ (where $\Gamma_0^N = \{g^N \mid g \in \Gamma_0\}$), due to the isomorphism $H^2(\mathbb{Z}/N\mathbb{Z}, \Gamma_0) \cong \Gamma_0/\Gamma_0^N$, it suffices to consider a set of representatives $\nu$ for this quotient group. Once $\nu$ is fixed, one calculates $m = \langle \gamma(\nu), \nu \rangle \in \ku^*$. As $\ku$ is algebraically closed of characteristic zero, the equation $q^N = m$ always possesses $N$ distinct solutions $q_0, q_1, \dots, q_{N-1}$ in $\ku^*$. Each of these roots $q_k$ defines a 3-cochain $\omega^{(q_k)}$ and consequently yields $N$ distinct associative zesting data $(\gamma, \lambda^{(\nu)}, \omega^{(q_k)})$ associated with the chosen representative $\nu$.
\end{remark}

\subsection{Concrete examples of zesting of pointed Hopf algebras}\label{subsec:concrete examples}

Before presenting concrete examples, let us outline the process for computing associative zesting data $(\gamma, \lambda, \omega)$ in the pointed case $H = \mathcal{B} \# \ku \Gamma$, based on the preceding results. We emphasize that constructing the zesting datum itself depends  only on the group $\Gamma$ and the structure of the Yetter-Drinfeld module $V$ associated with the skew-primitive generators $(x_i)_{1\leq i\leq \theta}$. Specifically, the module $V$ determines its support $\Supp(V)$ (and thus the universal grading group $G = \Gamma/\Supp(V)$) and the action of potential subgroups $\Gamma_0 \le \Supp(V)$ (needed to check the  condition for $\gamma$). 

Determining the appropriate $(\gamma, \lambda, \omega)$ and verifying the condition \eqref{assoczesting} does not require knowing a presentation (i.e., the defining relations) of the pre-Nichols algebra $\mathcal{B}$. Therefore, in the following examples, we will focus only on realizing the vector space $V$ as a Yetter-Drinfeld module over a group $\Gamma$ and computing the zestig data.

\subsubsection{Braidings of diagonal type and their realization}\label{subsec:diagonal_realization}

Many important Nichols algebras, particularly those relevant to the classification of pointed Hopf algebras over abelian groups, arise from braided vector spaces $(V, c)$ of \emph{diagonal type}. 

A braided vector space $(V, c)$ is of \emph{diagonal type} if there exists a basis $(x_i)_{i \in I}$ for $V$ and a matrix $\qf = (q_{ij})_{i,j \in I}$ of scalars $q_{ij} \in \ku^\times$ such that the braiding $c: V \otimes V \to V \otimes V$ on this basis is given by
\begin{align}\label{eq:diagonal_braiding_def}
c(x_i \otimes x_j) &= q_{ij} \, x_j \otimes x_i & \text{for all }& i, j \in I.
\end{align}
As the structure is determined by the \emph{braiding matrix} $\qf$, the  vector space is sometimes denoted by $V_{\qf}$ and the associated Nichols algebra by $\mathcal{B}_{\qf}$.

Braided vector spaces $V_{\qf}$ can be realized as Yetter-Drinfeld modules $\Gayd$ over an abelian group $\Gamma$. As the braiding in $\Gayd$ is given by $c_{V,W}(v \otimes w) = (v_{(-1)} \cdot w) \otimes v_{(0)}$, to realize $V_{\qf}$ as an object in $\Gayd$, we need to find:
\begin{enumerate}
    \item Group elements $(g_i)_{i \in I}$ in $\Gamma$,
    \item Characters $(\chi_j)_{j \in I}$ in $\widehat{\Gamma} = \Hom(\Gamma, \ku^\times)$,
\end{enumerate}
such that if we define the $\ku\Gamma$-coaction by $\rho(x_i) = g_i \otimes x_i$ and the $\ku\Gamma$-action by $h \cdot x_j = \chi_j(h) x_j$, the resulting Yetter-Drinfeld braiding $c(x_i \otimes x_j) = \chi_j(g_i) x_j \otimes x_i$ matches the desired braiding \eqref{eq:diagonal_braiding_def}. This requires the condition:
\begin{align}\label{eq:qij_realization_condition}
    q_{ij} = \chi_j(g_i) \quad \text{for all } i, j \in I.
\end{align}

\subsubsection{Associative zestings for $\qf$ of type $A(1|2)$}\label{ex:zesting_A12_revised}

Let $n>1$ be an integer and let $q \in \ku^\times$ be a primitive $n$-th root of unity. We consider the pointed Hopf algebra $H = \mathcal{B} \# \ku \Gamma$, where $\cB$ is a pre-Nichols algebra of super type $A(1|2)$ \cite{AnAn} and $\Gamma = C_2 \times C_{n^2} = \langle \alpha_1, \alpha_2 \rangle$. The 2-dimensional Yetter-Drinfeld module $V=\ku x_1 \oplus \ku x_2$ has braiding matrix $\qf = \begin{pmatrix} -1 & 1 \\ q^{-1} & q \end{pmatrix}$, realized via $g_1 = \alpha_1$, $g_2 = \alpha_2^n$, and appropriate characters $\chi_1, \chi_2 \in \widehat{\Gamma}$. The universal grading group is 
\begin{align*}
G &= U(H) \cong \Gamma / \Supp(V) \cong C_n, & \text{where }\Supp(V) &= \langle g_1, g_2 \rangle \cong C_2 \times C_n.
\end{align*}
Since $G=C_n$ is cyclic, we aim to apply Proposition \ref{prop:cyclic_zesting}.

\textbf{Data $(\Gamma_0, \Phi)$.}

We need a subgroup $\Gamma_0 \le \Supp(V)$ and a group homomorphism $\Phi: \Gamma_0 \to \widehat{\Gamma}$ satisfying the condition $\chi_j(g) = \Phi(g)(g_j)$ for all $g \in \Gamma_0, j=1,2$.

Suppose that $g_2=\alpha_2^n\in \Gamma_0$. Then condition \eqref{eq:central-subgr-condition} says that $q^{-1} = \Phi(g_2)(\alpha_1)=  \pm 1$, which is a contradiction to the fact that $q$ is a primitive $n$-th root.

Therefore, we take $\Gamma_0 = \langle g_1 \rangle = \langle \alpha_1 \rangle \cong C_2$, and search for group homomorphisms $\Phi: \Gamma_0 \to \widehat{\Gamma}$. This $\Phi$ is determined by the character $\chi' = \Phi(g_1) \in \widehat{\Gamma}$. Now condition \eqref{eq:central-subgr-condition} requires:
\begin{itemize}
    \item $\chi_1(g_1) = \chi'(g_1) \implies -1 = \chi'(\alpha_1)$.
    \item $\chi_2(g_1) = \chi'(g_2) \implies 1 = \chi'(\alpha_2^n)$.
\end{itemize}
Write $\chi' = \nu_1^a \nu_2^b$. By the above equalities, $a$ needs to be odd, so $a=1$, and $b$ needs to be a multiple of $n$, ie $b=nk$ for $k \in \{0, \dots, n-1\}$. Hence, potential characters are $\chi'_k = \nu_1 \nu_2^{nk}$.
Furthermore, $\Phi$ must be a group homomorphism of $\Gamma_0 \cong C_2$. This requires $2k \equiv 0 \pmod n$. Hence, the possible values for $k \in \{0, \dots, n-1\}$ satisfying $2k \equiv 0 \pmod n$ depend on the parity of $n$:
\begin{itemize}
    \item If $n$ is odd, then the only possible map is $\Phi_0$, where $\Phi_0(\alpha_1) = \nu_1$. 
    \item If $n$ is even, then there are two possibilities, either $\Phi_0(\alpha_1) = \nu_1$ or else $\Phi_{n/2}(\alpha_1) = \nu_1 \nu_2^{n(n/2)}$. 
\end{itemize}

\textbf{Applying Proposition \ref{prop:cyclic_zesting}.}

We will follow the notation from Proposition \ref{prop:cyclic_zesting}. We must use $\nu = \alpha_1 \in \Gamma_0$.

\begin{itemize}[leftmargin=*]\renewcommand{\labelitemi}{$\circ$}
\item Using the map $\Phi_0$ we have $m =  \Phi_0(\alpha_1)(\alpha_1) = \nu_1(\alpha_1) = -1$. We need $s \in \ku^*$ such that $s^n = -1$. Set $\lambda = \lambda^{(g_1)}$ and $\omega = \omega^{(s)}$. By Proposition \ref{prop:cyclic_zesting}, $(\gamma_0, \lambda^{(g_1)}, \omega^{(s)})$ is an associative $C_n$-zesting datum. There are $n$ choices for $s$.

\item If $n$ is even and we use the map $\Phi_{n/2}$, then 
\begin{align*}
m &= \langle \gamma_{n/2}(\alpha_1), \alpha_1 \rangle = \Phi_{n/2}(\alpha_1)(\alpha_1) = (\nu_1 \nu_2^{n(n/2)})(\alpha_1) = \nu_1(\alpha_1) \nu_2^{n(n/2)}(\alpha_1) =-1.
\end{align*}
Again, we need $s^n = -1$.  Proposition \ref{prop:cyclic_zesting} gives $n$ zestings $(\gamma_{n/2}, \lambda^{(g_1)}, \omega^{(s)})$.
\end{itemize}

In summary, for $\Gamma_0 = \langle g_1 \rangle$, non-trivial associative zestings can be constructed using Proposition \ref{prop:cyclic_zesting}. If $n$ is odd, there is one family (determined by $\Phi_0$) of $n$ zestings. If $n$ is even, there are two families (determined by $\Phi_0$ and $\Phi_{n/2}$), each one has $n$ zestings.

\subsubsection{Associative zestings for Fomin-Kirillov algebras}
Now we consider a family of examples of pointed Hopf algebras of non-diagonal type.

The Fomin-Kirillov algebra $\mathcal{FK}_3$ is the Nichols algebra $\mathcal{B}(V)$ of the 3-dimensional braided vector space $V = \ku\{x_0, x_1, x_2\}$ whose braiding is $c(x_i \otimes x_j) = -x_{2i-j \pmod 3} \otimes x_i$, which is non-diagonal and can be realized over the symmetric group $\mathbb{S}_3$. More generally \cite{GI-Sanchez}, we can realize $V$ is realized as a Yetter-Drinfeld module $V_k$ over the group algebra $H = \ku\Gamma$, where $\Gamma = \mathbb{G}_{3,\ell} = \langle s, t \mid s^3 = t^{2\ell}=1, ts=s^2t \rangle$ ($\ell\in\mathbb N$ can be chosen arbitrarily) and the structure depends on $k \in \{0, \dots, \ell-1\}$ as follows:
\begin{itemize}[leftmargin=*]
\item The coaction $\rho: V_k \to H \otimes V_k$ is given by
\begin{align*}
\rho(x_i) &= g_i \otimes x_i, & \text{where } &g_i := s^i t^{2k+1} \quad \text{for } i \in \{0, 1, 2\}.
\end{align*}

\item The  action of each $g\in\Gamma$ is given by 
\[ g \cdot x_i = (-1)^r x_{g \cdot i}, \]
where $g\cdot i\in \{0, 1, 2\}$ is given by
\[ g \cdot i := j \quad \text{if and only if} \quad g g_i g^{-1} = g_{j}. \]

\end{itemize}

Let $N=2k+1$, $d=\gcd(N, 2\ell)$.
The support and the universal grading group are
\begin{align*}
\Supp(V_k) &= \langle s, t^N \rangle, &
U(H) &\cong \Gamma / \Supp(V_k) \cong C_d.
\end{align*}
For further simplicity, we impose the slightly stronger condition that $N$ divides $\ell$. Thus, we proceed with $G = U(H) \cong C_N$.

First, we identify an appropriate subgroup $\Gamma_0$, which must be central and contained in $\Supp(V_k) = \langle s, t^N \rangle$. The center of $\Gamma = \mathbb{G}_{3,\ell}$ is $Z(\Gamma) = \langle t^2 \rangle$.  Then we can choose
$$\Gamma_0 = Z(\Gamma) \cap \Supp(V_k) = \langle t^{2\gcd(N,\ell)} \rangle = \langle t^{2N} \rangle \cong C_{\ell/N}.$$

Next, we look for group homomorphisms $\Phi: \Gamma_0 \to \widehat{\mathbb{G}}_{3,\ell}$ that satisfy the centrality condition in Proposition \ref{prop:relative_r_pointed}, namely 
\begin{align*}
g \cdot x_j &= \Phi(g)(g_j) x_j & \text{for all }g \in \Gamma_0, & j\in\{0,1,2\}.
\end{align*}
Recall that $\widehat{\mathbb{G}}_{3,\ell} = \langle \chi \rangle \cong C_{2\ell}$, where $\chi(s)=1$ and $\chi(t)=\zeta_{2\ell}$. An homomorphism $\Phi: \Gamma_0 \to \widehat{\mathbb{G}}_{3,\ell}$ is determined by the image of the generator $g_0 = t^{2N}$, say $\Phi(g_0) = \chi^w$.
Now $w$ must hold the following two conditions:
\begin{enumerate}
\item $\Phi$ must be a group homomorphism. This requires 
\begin{align*}
\Phi(g_0)^{|\Gamma_0|} &= \Phi(g_0^{|\Gamma_0|}) = \Phi(e) = 1
&&\implies &
(\chi^w)^{\ell/N} &= \chi^{w\ell/N} \equiv  1.
\end{align*}
Thus, $2\ell$ must divide $w\ell/N$.
\item $\Phi$ must satisfy the centrality condition: This requires $\Phi(g_0)(g_j) = 1$, because $g_0$ is central and then $g_0 \cdot x_j = x_j$. Thus, $\chi^w(s^j t^N) = (\chi(t)^N)^w = \zeta_{2\ell}^{Nw} = 1$, which implies that $2\ell$ divides $Nw$.
\end{enumerate}
The integers $w$ (modulo $2\ell$) satisfying both conditions are precisely the multiples of $2\ell/d$, where $d = \gcd(N, \ell/N) = \gcd(2k+1, \ell/(2k+1))$.
Thus, there exist $d$ possible homomorphisms $\Phi$ satisfying the requirements, denoted $\Phi_a$ for $a \in \{0, 1, \dots, d-1\}$:
\[ \Phi_a(t^{2N}) = \chi^{a \cdot (2\ell/d)}.\]
Each $\Phi_a$ determines a relative $r$-form via the map $\gamma_a: \Gamma_0 \to \Alg(H,\ku)$.

Now we assemble the full associative zesting datum $(\gamma, \lambda, \omega)$ using Proposition \ref{prop:cyclic_zesting}. We need to choose $\gamma$, $\lambda$, and $\omega$.

\begin{itemize}[leftmargin=*]\renewcommand{\labelitemi}{$\triangleright$}
\item \emph{Choose $\gamma$}: Select one of the $d = \gcd(N, \ell/N)$ possible centrality maps, $\gamma = \gamma_a$, corresponding to $\Phi_a(t^{2N}) = \chi^{a \cdot (2\ell/d)}$ for $a \in \{0, 1, \dots, d-1\}$.

    \item \emph{Choose $\lambda$:} Select a 2-cocycle $\lambda \in Z^2(G, \Gamma_0) = Z^2(C_N, C_{\ell/N})$. The cohomology classes are classified by $H^2(C_N, \Gamma_0) \cong \Gamma_0/\Gamma_0^N$. The order of this group is $d' = |\Gamma_0/\Gamma_0^N| = \gcd(N, |\Gamma_0|) = \gcd(N, \ell/N) = d$. We choose $\lambda$ to be the standard cocycle $\lambda = \lambda^{(\nu_s)}$ defined by \eqref{eq:cyclic_lambda2_nu_def_final}, where $\nu_s \in \Gamma_0$ is a representative for one of the $d$ distinct cosets of $\Gamma_0^N$ in $\Gamma_0$. Let $s \in \{0, 1, \dots, d-1\}$ index this choice. A standard choice is $\nu_s = (t^{2N})^s$.

    \item \emph{Determine $m$ and find $q$}: Calculate the scalar $m = m_{a,s} := \langle \gamma_a(\nu_s), \nu_s \rangle$. Substituting $\nu_s = (t^{2N})^s$ and $\gamma_a \leftrightarrow \Phi_a \leftrightarrow \chi^{w_a}$ with $w_a = a \cdot 2\ell/d$:
        $$ m_{a,s} = \langle \gamma_a( (t^{2N})^s ), (t^{2N})^s \rangle = \chi^{w_a}((t^{2N})^s) = (\chi(t))^{w_a \cdot 2Ns} = (\zeta_{2\ell})^{ (a \cdot 2\ell/d) \cdot 2Ns } $$
        $$ m_{a,s} = \zeta_{d}^{a N s}  $$
        We now need to find $q \in \ku^*$ such that $q^N = m_{a,s} = \zeta_{d}^{aNs}$. There are $N$ distinct solutions for $q$. Let $q_{a,s,j}$ (for $j=0, \dots, N-1$) denote these $N$ roots.

    \item \emph{Choose $\omega$:} Set $\omega = \omega^{(q_{a,s,j})}$ using the standard 3-cochain defined in \eqref{eq:cyclic_omega3_nu_def} .
\end{itemize}

By Proposition \ref{prop:cyclic_zesting}, the triple $(\gamma_a, \lambda^{(\nu_s)}, \omega^{(q_{a,s,j})})$ forms a valid associative $G$-zesting datum for $H$.

In summary, for $H = \mathcal{FK}_3 \# \ku \mathbb{G}_{3,\ell}$ (assuming $N=2k+1$ divides $\ell$), we identified $d = \gcd(N, \ell/N)$ possible centrality maps $\gamma_a$. For each $\gamma_a$ and each of the $d$ cohomology classes for $\lambda$ (represented by $\lambda^{(\nu_s)}$), we calculated the scalar $m_{a,s} = \zeta_{d}^{aNs}$. We can always find $N$ distinct $N$-th roots $q$ of $m_{a,s}$. Each resulting triple $(\gamma_a, \lambda^{(\nu_s)}, \omega^{(q)})$ constitutes a valid associative $C_N$-zesting datum. Each datum defines a coquasi-Hopf algebra $H^\lambda$ via Theorem \ref{thm:pointed_zesting_datum}.

\subsection{Braided zestings of pointed Hopf Algebras}

We now turn to braided zestings for pointed Hopf algebras. This requires starting with a finite-dimensional coquasitriangular pointed Hopf algebra $(H, r)$, where $r: H \otimes H \to \ku$ is the universal $r$-form. A necessary condition for $H$ to admit such a structure is that its group of group-like elements, $\Gamma = G(H)$, must be abelian. Following the classification of coquasitriangular structures on pointed Hopf algebras \cite{bontea2017pointed}, it is enough to consider $H= \mathcal{B}(V) \# \ku\Gamma$, where $V$ and $r$ are attached to a triple $(r_0,V,r_1)$ as follows:
\begin{itemize}
\item $r_0: \Gamma \times \Gamma \to \ku^*$ is a bicharacter,
\item $V \in \mathcal{Z}_{sym}(\cC(\Gamma, r_0))_-$, which means that $V\in\Gayd$ has a basis $(x_i)_{i\in I}$ with coaction on $x_i$ given by $g_i\in\Gamma$ such that $r_0(g_i,-)=r_0(-,g_i)^{-1}$, and action given by $\chi_i=r_0(g_i,-)\in\widehat{\Gamma}$ such that $\chi_i(g_i)=-1$.
\item $r_1:V\otimes V \to \ku$ a morphism of $\Gamma$-comodules.
\end{itemize}
Here, $r_0$ and $r_1$ correspond to the restrictions of $r$ to $\Gamma\times\Gamma$ and $V\otimes V$, respectively.

\begin{proposition}\label{prop:g_is_character_inducing}
Let $H = \mathcal{B}(V)\#\Bbbk\Gamma$ be a finite-dimensional coquasitriangular pointed Hopf algebra with universal $r$-form $r$ constructed from the triple $(r_0, V, r_1)$.

Then any group-like element $g \in \Gamma$ is character-induced with respect to this $U(H)$-grading. The corresponding group character $\lambda_g: U(H) \to \Bbbk^*$ is given by 
\begin{align*}
\lambda_g(s) &= r_0(\gamma_s, g)r_0(g, \gamma_s), &&\text{where }\gamma_s \in \Gamma \text{ is any representative of }s \in U(H).
\end{align*}
\end{proposition}

\begin{proof}
An element $g \in \Gamma$ is \emph{$\pi$-character-inducing} (Definition~\ref{def:pi_character_inducing_grouplike}) if
\begin{align*}
r(x_{(1)},g)r(g,x_{(2)}) &= \lambda_g(s)\epsilon(x)  &&\text{for all }x \in H_s, 
\end{align*}
where $\lambda_g: U(H) \to \Bbbk^*$ is the specific group character defined in the proposition statement.

Let $\lambda_g: U(H) \to \Bbbk^*$ be defined by $\lambda_g(s) := r_0(\gamma_s, g)r_0(g, \gamma_s)$, where $\gamma_s \in \Gamma$ is an arbitrary representative of $s \in U(H)$. It is well-defined since $V \in \mathcal{Z}_{sym}(\mathcal{C}(\Gamma, r_0))_-$  implies that $r_0(k,g)r_0(g,k)=1$ for all $k \in \Supp(V)$. The map $\lambda_g$ is a group homomorphism since  $r_0$ is bicharacter.
Now we want to verify that 
\begin{align}\label{eq:character-induced-example}
r(x_{(1)},g)r(g,x_{(2)}) &= \lambda_g(s)\epsilon(x) &\text{for all }x \in H_s, s\in U(H).
\end{align}

Recall that the $s$-homogeneous component $H_s$ consists of linear combinations of elements $x=b\#\gamma$, where $b \in \mathcal{B}(V)$ and $\gamma \in \Gamma$ is a representative of $s \in U(H)$. 
As $\epsilon(x)=r(x,y)=r(y,x)=0$ when $b\in \mathcal{B}^n(V)$ for $n>0$ (for all $y\in H$), \eqref{eq:character-induced-example} trivially holds in this case, so it remains to consider $b=1$. In this case $x=1\# \gamma$, so $x_{(1)}=x_{(2)}=1\# \gamma$, and \eqref{eq:character-induced-example} holds by definition.
\end{proof}

Due to the previous results we get the following characterization for braided $U(H)$-zestings of $H$:

\begin{corollary}
Braided zestings of $H$ correspond to 5-uples $(\Gamma_0,\Phi,\lambda, \Phi, \omega, t)$, where 
\begin{enumerate}[leftmargin=*, label=(\roman*)]
\item $\Gamma_0$ is a subgroup of $\Supp V$ and $\Phi: \Gamma_0 \to \widehat{\Gamma}$ is a group homomorphism such that 
\begin{align*}
r_0(g,g_i)^{-1}&= \langle g_i, \Phi(g) \rangle && \text{for all }g\in \Gamma_0, 1\le i\le \theta.    
\end{align*}
\item $(\Phi, \lambda, \omega)$ is an associative $U(H)$-zesting, and 
\item $t\in C^2(U(H),\ku^*)$ is such that equations \eqref{eq:BZ2} and \eqref{eq:BZ3} holds, that is:
\begin{align*}
\frac{\omega(g_1, g_2, g_3) \omega(g_2, g_3, g_1)}{\omega(g_2, g_1, g_3)} &= \frac{t(g_1, g_2) t(g_1, g_3)}{t(g_1, g_2 g_3)},
\\
r_0(\lambda(g_1,g_2), g_3) r_0(g_3, \lambda(g_1,g_2)) &= \frac{\omega(g_1, g_2, g_3) \omega(g_3, g_1, g_2)}{\omega(g_1, g_3, g_2)} \frac{t(g_1, g_3) t(g_2, g_3)}{t(g_1 g_2, g_3)},
\end{align*}
for all $g_1, g_2, g_3 \in U(H)$.
\end{enumerate}
\end{corollary}
\qed

\begin{example}
Let $\Gamma= \Z/4\Z$. Fix a generator $\sigma$ of $\Gamma$ and $r_0:\Z/4\Z \times \Z/4\Z \to \ku^*$ be the bicharacter such that $r_0(\sigma,\sigma)=i$. Let $V$ be a $\Z/4\Z$-graded module such that $V \in \mathcal{Z}_{sym}(\cC(\Z/4\Z,r_0))_-$. In this case $\Supp (V)= \left\langle g\right\rangle\cong \Z/2\Z$, where $g:= \sigma^2\in \Z/4\Z$, and following Proposition \ref{Universal grading pointed}, $U(H)\cong \Z/2\Z$.

Then there exists a braided zesting datum given by:
\begin{itemize}
\item The group homomorphism $\Phi: \Z/2\Z \to \widehat{\Z/4\Z}$ such that $\left\langle \sigma, \Phi(g)\right \rangle=-1$.
\item The non-trivial 2-cocycle $\lambda \in H^2(\Z/2\Z, \Z/2\Z)$ such that $\lambda(1,1)=g$.
\item The 3-cochain $\omega \in C^3(\Z/2,\ku^*)$ with $\omega(1,1,1)= \zeta$, where $\zeta$ denotes a 4th-root of unity.
\item The 2-cochain $t\in C^2(\Z/2\Z, \ku^*)$ with $t(1,1)= \eta$ where $\eta^2=\zeta$.
\end{itemize}

\end{example}


\bibliographystyle{alpha}
\bibliography{references}

\end{document}